\documentclass{amsart}
\usepackage{amsthm,xcolor, bm}
\usepackage{latexsym,amssymb,amsfonts,amsmath,graphicx, url,mathtools}
\usepackage[mathscr]{eucal}
\usepackage[vcentermath,enableskew]{youngtab}
\usepackage{color}
\usepackage[noadjust]{cite} % the spacing of package cite solves a problem we don't have and introduces a problem we do.
\usepackage{tikz,hyperref}
\usepackage[margin=1in]{geometry}
\usepackage{enumerate,textcomp, wasysym, enumitem}

\newtheorem{theorem}{Theorem}[section]
\newtheorem{proposition}[theorem]{Proposition}
\newtheorem{lemma}[theorem]{Lemma}
\newtheorem{conjecture}[theorem]{Conjecture}
\newtheorem{corollary}[theorem]{Corollary}

\theoremstyle{definition}
\newtheorem{definition}[theorem]{Definition}
\newtheorem{example}[theorem]{Example}
\newtheorem{remark}[theorem]{Remark}

\newtheoremstyle{named}{}{}{\itshape}{}{\bfseries}{.}{.5em}{#1 \thmnote{#3}}
\theoremstyle{named}
\newtheorem*{namedtheorem}{Theorem}

\newcommand{\sort}{\textup{sort}}

\def\fb{\mathrm{fb}}

\def\l{\mathrm{\ell}}

\title{An Orthodontia Formula for Grothendieck Polynomials}

\author{Karola M\'esz\'aros}
\address{Karola M\'esz\'aros, Department of Mathematics, Cornell University, Ithaca, NY 14853.  \newline\textup{karola@math.cornell.edu}
}

\author{Linus Setiabrata}
\address{Linus Setiabrata, Department of Mathematics, University of Chicago, Chicago, IL, 60637.  \newline\textup{linus@math.uchicago.edu}
}

\author{Avery St.~Dizier}
\address{Avery St. Dizier, Department of Mathematics, University of Illinois at Urbana-Champaign, Urbana, IL 61801.  \newline\textup{stdizie2@illinois.edu}
}

\thanks{Karola M\'esz\'aros received support from  CAREER NSF Grant DMS-1847284. Avery St.~Dizier received support from NSF Grant DMS-2002079.}


\begin{document}
	\maketitle
	\begin{abstract}
		We give a new operator formula for Grothendieck polynomials that generalizes Magyar's Demazure operator formula for Schubert polynomials. 
		Our proofs are purely combinatorial, contrasting with the geometric and representation theoretic tools used by Magyar.
		We apply our formula to prove a necessary divisibility condition for a monomial to appear in a given Grothendieck polynomial.
	\end{abstract}
		
	\section{Introduction}
	
	Schubert polynomials $\mathfrak{S}_w$ and Grothendieck polynomials $\mathfrak{G}_w$ are multivariate polynomials associated to permutations $w\in S_n$. Schubert (resp.\ Grothendieck) polynomials were introduced by Lascoux and Sch\"utzenberger in \cite{LS1,LS2} as a set of distinguished representatives for the cohomology (resp.\ K-theoretic) classes of Schubert cycles in the cohomology ring (resp.\ K-theory) of the flag variety of $\mathbb{C}^n$. Since their introduction, Schubert polynomials have become central objects in algebraic combinatorics. Their rich combinatorial structure is evident from the myriad formulas, such as \cite{laddermoves, BJS, FKschub, nilcoxeter, thomas, lenart, prismtableaux,balancedtableaux}. Many formulas for Schubert polynomials generalize to Grothendieck polynomials.
%	, such as \cite{grothtransition,grothmonks,FKgroth,grobner}. 
	Recent work \cite{grothlattice,grothice,grothcoloredlattice} has uncovered novel formulas for Grothendieck polynomials and their generalizations.	
	
%	Recently, there has been a push toward elucidating the combinatorics of the supports and Newton polytopes of significant polynomials in algebraic combinatorics, including Schur, Schubert, Grothendieck and key polynomials \cite{FMS,MTY,FGkey,BGHMOP}. The theory of Lorentzian polynomials...

	In this paper, we focus on the following algebraic formula for Schubert polynomials due to Magyar \cite{magyar}:
	\begin{equation} \label{magyar} \mathfrak{S}_w = \omega_1^{k_1}\cdots\omega_n^{k_n}\pi_{i_1}(\omega_{i_1}^{m_1} \pi_{i_2}(\omega_{i_2}^{m_2}\cdots \pi_{i_\l}(\omega_{i_\l}^{m_\l})\cdots )). 
	\end{equation}
	The formula uses combinatorial data 
	\[\bm{i}(w)=(i_1,\ldots,i_\l),\quad \bm{k}(w)=(k_1,\ldots,k_n),\quad \mbox{and}\quad \bm{m}(w)=(m_1,\ldots,m_\l)\]
	associated to Rothe diagrams to write Schubert polynomials in terms of the Demazure operators $\pi_j$ and fundamental weights $\omega_j=x_1\cdots x_j$. Unlike the usual recursive definition of Schubert polynomials through divided difference operators, Magyar's formula is ``ascending'': the degree weakly increases at each step of the formula. We generalize Magyar's formula to Grothendieck polynomials, consequently giving a new proof for Schubert polynomials in the process. We now state our main theorem; for the necessary definitions see Section \ref{sec:background}.
	
	\begin{theorem}
		\label{thm:groth-for-intro} 		
		For  any $w\in S_n$,  the Grothendieck polynomial 	 $\mathfrak{G}_w$ is given by
			\[
			 \mathfrak{G}_w = \omega_1^{k_1}\cdots\omega_n^{k_n}\overline{\pi}_{i_1}(\omega_{i_1}^{m_1}\overline{\pi}_{i_2}(\omega_{i_2}^{m_2}\cdots\overline{\pi}_{i_\l}(\omega_{i_\l}^{m_\l})\cdots)),
			\]
			where $(\bm{i}(w), \bm{k}(w), \bm{m}(w))$ is the orthodontic sequence of $w$, $\omega_i$ is the fundamental weight $\omega_i=x_1\cdots x_i$, and $\overline{\pi}_j(f) = \pi_j((1-x_{j+1})f)$.
	 \end{theorem}
	
	 Our proof of Theorem \ref{thm:groth-for-intro}  is purely combinatorial, and yields a combinatorial proof of \eqref{magyar} as well -- contrasting with the geometric and representation theoretic tools used in \cite{magyar}. We apply Theorem \ref{thm:groth-for-intro} and the inductive tools developed for its proof to derive Theorem \ref{thm:upwardsdivisibility-intro}, a new divisibility restriction for monomials appearing in a Grothendieck polynomial. We refer to Section \ref{sec:application} for notation and details.
	 
	 \begin{theorem}
	 	\label{thm:upwardsdivisibility-intro}
	 	For any permutation $w\in S_n$, all monomials appearing in $\mathfrak{G}_w$ divide $\bm{x}^{\overline{D(w)}}$.
	 \end{theorem}
	 
	\subsection*{Outline of this paper}
	
	Section~\ref{sec:background} gives background on Schubert and Grothendieck polynomials, and explains the machinery behind Magyar's orthodontia formula. In Section \ref{sec:sortpermgroth}, we define the class of sorted permutations and introduce a projection onto this class. In Section \ref{sec:orthobruhat}, we track changes in orthodontia upon sorting or moving up in weak Bruhat order. In Section \ref{sec:grothortho}, we construct some inductive tools and prove Theorem \ref{thm:groth-for-intro}. In Section \ref{sec:application}, we apply Theorem \ref{thm:groth-for-intro} to study the supports of Grothendieck polynomials, proving Theorem \ref{thm:upwardsdivisibility-intro}.  We conclude with a brief discussion of strongly-separated diagrams in Section \ref{sec:stronglysep}.

	\section{Background}
	\label{sec:background}
	\subsection{Conventions}
	For $m,n\in \mathbb{N}$, we use the notation $[m,n]$ to mean the set $\{m,m+1,\ldots,n \}$, and the notation $[n]$ to mean $\{1,2,\ldots,n \}$. For $j\in [n-1]$, $s_j$ will denote the adjacent transposition in the symmetric group $S_n$ swapping $j$ and $j+1$. Throughout, we will take permutations as acting on the right, switching positions, not values. For example $ws_1$ equals $w$ with the numbers $w(1)$ and $w(2)$ swapped.

	\subsection{Difference Operators on Polynomials}
	We recall the definitions of four types of operators on polynomial rings.
	\begin{definition}
		Fix any $n\geq 0$. The \emph{divided difference operators} $\partial_j$ for $j\in[n-1]$ are operators on the polynomial ring $\mathbb{C}[x_1,\ldots,x_n]$ defined by
		\[\partial_j(f)=\frac{f-(s_j\cdot f)}{x_j-x_{j+1}}
		=\frac{f(x_1,\ldots,x_n)-f(x_1,\ldots,x_{j-1},x_{j+1},x_j,x_{j+2},\ldots,x_n)}{x_j-x_{j+1}}.\]
		The \emph{Demazure operators} $\pi_j$, the \emph{isobaric divided difference operators} $\overline{\partial}_j$, and the \emph{Demazure--Lascoux operators} $\overline{\pi}_j$ are defined on $\mathbb{C}[x_1,\ldots,x_n]$ respectively by
		\begin{align*}
			\pi_j(f) &= \partial_j(x_j f),\\
			\overline{\partial}_j(f)&=\partial_j((1-x_{j+1})f),\\
			\overline{\pi}_j(f) &= \partial_j(x_j(1-x_{j+1})f).
		\end{align*}
	\end{definition}
	
	The following lemmas collect several basic properties of divided and isobaric divided difference operators which will be used frequently.
	\begin{lemma}
		The divided difference operators satisfy the following properties.
		\begin{itemize}
			\item $\partial_j\partial_j=0$ for all $j$.
			\item $\partial_j\partial_k=\partial_k\partial_j$ whenever $|j-k|>1$.
			\item $\partial_j\partial_{j+1}\partial_j=\partial_{j+1}\partial_j\partial_{j+1}$.
			\item $\partial_j(f) = 0$ if and only if $f$ is symmetric in $x_j$ and $x_{j+1}$.
			\item If $\partial_j(f)=0$, then $\partial(fg)=f\partial(g)$.
		\end{itemize}
	\end{lemma}

	\begin{lemma}
		The isobaric divided difference operators satisfy the following properties.
		\begin{itemize}
			\item $\overline{\partial}_j\overline{\partial}_j=\overline{\partial}_j$ for all $j$.
			\item $\overline{\partial}_j\overline{\partial}_k=\overline{\partial}_k\overline{\partial}_j$ whenever $|j-k|>1$.
			\item $\overline{\partial}_j\overline{\partial}_{j+1}\overline{\partial}_j=\overline{\partial}_{j+1}\overline{\partial}_j\overline{\partial}_{j+1}$.
			\item $\overline{\partial}_j(f)$ is symmetric in $x_j$ and $x_{j+1}$.
%			\item $\overline{\partial}_j(f) = 0$ whenever $f$ is symmetric in $x_j$ and $x_{j+1}$.
%			\item If $\overline{\partial}_j(f)=0$, then $\partial(fg)=f\partial(g)$.
		\end{itemize}
	\end{lemma}

	\subsection{Schubert and Grothendieck Polynomials}
	
	\begin{definition}
		The \emph{Schubert polynomial} $\mathfrak{S}_w$ of $w\in S_n$ is defined recursively on the weak Bruhat order. Let $w_0=n \hspace{.1cm} n\!-\!1 \hspace{.1cm} \cdots \hspace{.1cm} 2 \hspace{.1cm} 1 \in S_n$, the longest permutation in $S_n$. If $w\neq w_0$ then there is $j\in [n-1]$ with $w(j)<w(j+1)$ (called an \emph{ascent} of $w$). The polynomial $\mathfrak{S}_w$ is defined by
		\begin{align*}
		\mathfrak{S}_w=\begin{cases}
		x_1^{n-1}x_2^{n-2}\cdots x_{n-1}&\mbox{ if } w=w_0,\\
		\partial_j \mathfrak{S}_{ws_j} &\mbox{ if } w(j)<w(j+1).
		\end{cases}
		\end{align*}
	\end{definition}
	
	\begin{definition}
		The \emph{Grothendieck polynomial} $\mathfrak{G}_w$ of $w\in S_n$ is defined analogously to the Schubert polynomial,
		with
		\begin{align*}
		\mathfrak{G}_w=\begin{cases}
		x_1^{n-1}x_2^{n-2}\cdots x_{n-1}&\mbox{ if } w=w_0,\\
		\overline{\partial}_j \mathfrak{G}_{ws_j} &\mbox{ if } w(j)<w(j+1).
		\end{cases}
		\end{align*}	
	\end{definition}
	\begin{proposition}
		\label{prop:repeatoperators}
		Let $w\in S_n$ with $w(j)<w(j+1)$. Then
		\[\partial_j(\mathfrak{G}_w)=0 \quad \mbox{and}\quad \overline{\partial}_j(\mathfrak{G}_w)=\mathfrak{G}_w.\]
	\end{proposition}
	\begin{proof}
		The conclusions follow readily from the basic properties of $\partial_k$ and $\overline{\partial}_k$, together with the recursive definition of $\mathfrak{G}_w$.
	\end{proof}
	
	It can be seen from the recursive definitions that $\mathfrak{S}_w$ is homogeneous of degree equal to the number of inversions of $w$, and equals the lowest-degree nonzero homogeneous component of $\mathfrak{G}_w$. See \cite{manivel} for a deeper introduction to Schubert polynomials.
	
	\subsection{Orthodontia of Diagrams}
	We describe the orthodontia algorithm for diagrams due to Magyar in \cite{magyar}. We closely follow the exposition of \cite{zeroone}.
		
	By a \emph{diagram}, we mean a subset $D\subseteq [n]^2$, the $n\times n$ grid. We view $D$ from a column perspective as $D=(C_1,C_2,\ldots,C_n)$, where each $C_j$ is a subset of $[n]$. The subsets $C_j$ are naturally called the \emph{columns} of $D$. Graphically, we draw $D$ as a collection of boxes $(i,j)$ in a grid, viewing an element $i\in C_j$ as a box in row $i$ and column $j$ (reading the indices in the same way as matrix notation). There is a canonical diagram associated to any permutation.
	
	\begin{definition}
		The \emph{Rothe diagram} $D(w)$ of a permutation $w\in S_n$ is the diagram
		\[ D(w)=\{(i,j)\in [n]^2 \mid i<w^{-1}(j)\mbox{ and } j<w(i) \}. \]
		$D(w)$ can be visualized as the set of boxes left in the $n\times n$ grid 
		after you cross out all boxes weakly below $(i,w(i))$ in the same column, or weakly right of $(i,w(i))$ in the same row for each $i\in [n]$.
	\end{definition}
	
	\begin{example}
		If $w=31542$, then 
		\begin{center}
			\includegraphics[scale=1]{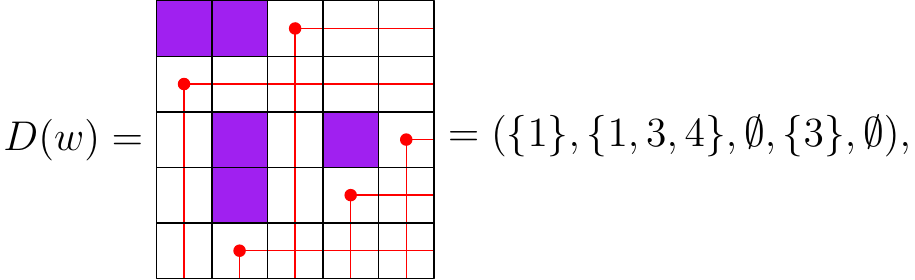}
		\end{center}
		where we indicate the boxes removed with red lines.
	\end{example}
	
	We now explain Magyar's orthodontia algorithm. For a column $C\subseteq [n]$, let the \emph{multiplicity} $\mathrm{mult}_D(C)$ be the number of columns of $D$ which are equal to $C$. Let $D$ be the Rothe diagram of a permutation $w\in S_n$ with columns $C_1,C_2,\ldots,C_n$. We describe an algorithm to produce vectors
	\[\bm{i}(w)=(i_1,\ldots,i_\l),\quad \bm{k}(w)=(k_1,\ldots,k_n),\quad \mbox{and}\quad \bm{m}(w)=(m_1,\ldots,m_\l)\]
	from $D$. To begin the first step, for each $j\in [n]$ let $k_j=\mathrm{mult}_D([j])$, the number of columns of $D$ of the form $[j]$. Replace all such columns by empty columns for each $j$ to get a new diagram $D_-$.

	Given a column $C\subseteq [n]$, a \emph{missing tooth} of $C$ is a positive integer $i$ such that $i\notin C$, but $i+1\in C$. The only columns without missing teeth are the empty column and the intervals $[i]$. Hence the first nonempty column of $D_-$ (if there is any) contains a smallest missing tooth $i_1$. Switch rows $i_1$ and $i_1+1$ of $D_-$ to get a new diagram $D'$.
	
	In the second step, repeat the above with $D'$ in place of $D$. Specifically, let $m_1=\mathrm{mult}_{D'}([i_1])$ and replace all columns of the form $[i_1]$ in $D'$ by empty columns to get a new diagram $D_-'$. Find the smallest missing tooth $i_2$ of the first nonempty column of $D_-'$, and switch rows $i_2$ and $i_2+1$ of $D_-'$ to get a new diagram $D''$. Continue in this fashion until no nonempty columns remain. 
	
	%	It is easily seen that the sequences $\bm{i}=(i_1,\ldots,i_\l)$ and $\bm{m}=(k_1,\ldots,k_n;\,m_1,\ldots,m_\l)$ just constructed have the desired properties.
	
	\begin{definition}
		\label{def:ikmsequence}
		The triple $(\bm{i}(w),\bm{k}(w),\bm{m}(w))$ constructed in the preceding algorithm is called the \emph{orthodontic sequence} of $w$.
	\end{definition}
	
	\begin{example}
		If $w=31542$, then the orthodontic sequence algorithm produces the diagrams shown in Figure \ref{fig:orthoexample}. The sequence of missing teeth gives $\bm{i}(w)=(2,3,1)$, $\bm{k}(w)=(1,0,0,0,0)$, and $\bm{m}=(0,1,1)$.
	\end{example}

	\begin{figure}
		\begin{center}
			\includegraphics[scale=.85]{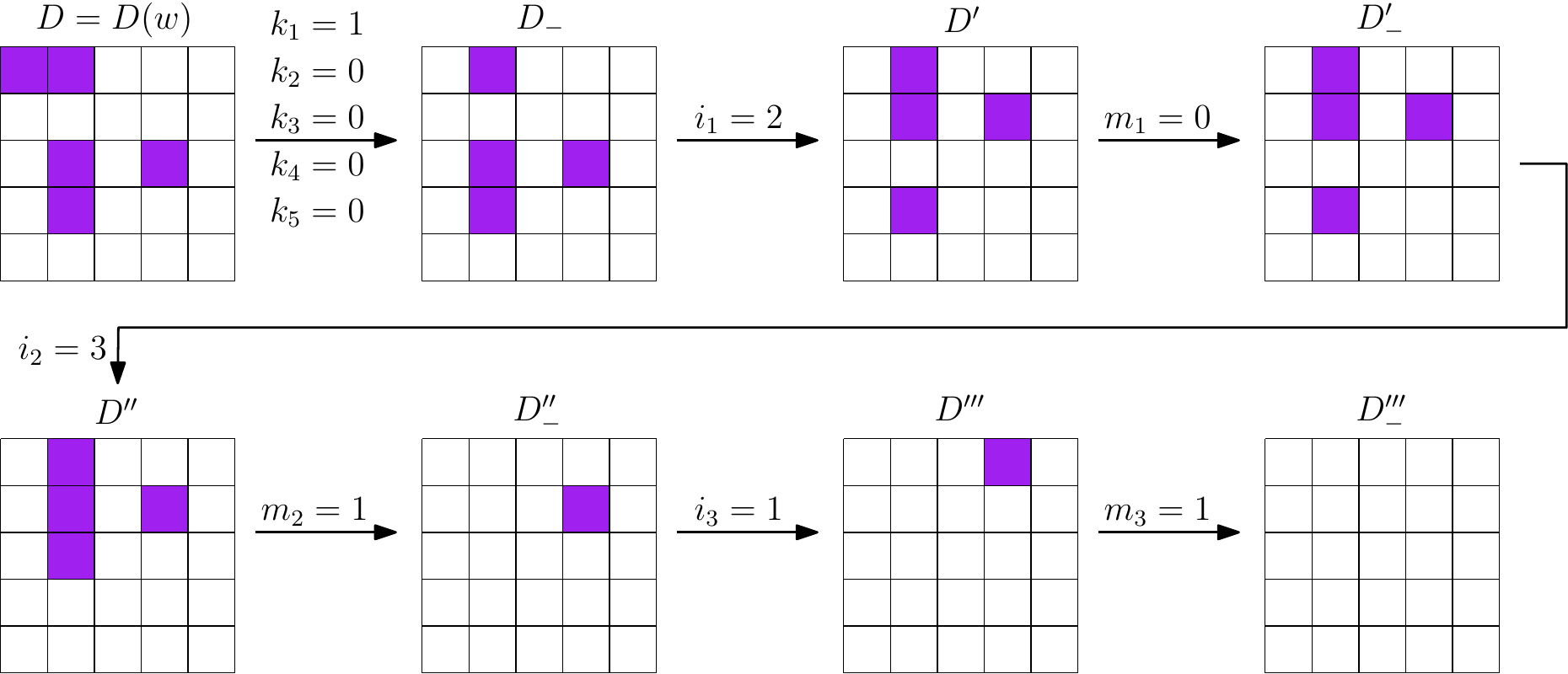}
		\end{center}
	\caption{Execution of the orthodontic sequence algorithm on $D(w)$ for $w=31542$.}
	\label{fig:orthoexample}
	\end{figure}
	
	\begin{remark}
		Magyar's orthodontia algorithm applies more generally to any strongly separated diagram $D$, after possibly changing the order of the columns. See Section \ref{sec:stronglysep} for a brief summary, or \cite{magyar} for further details.
	\end{remark}

	Our interest in orthodontia stems from the following orthodontic operator formula for Schubert polynomials, which we generalize to Grothendieck polynomials in Section \ref{sec:grothortho}.
	\begin{theorem}[{\cite[Proposition~15]{magyar}}]
		\label{thm:magyaroperatortheorem}
		Let $w\in S_n$ have orthodontic sequence 
		\[\bm{i}(w)=(i_1,\ldots,i_\l),\quad \bm{k}(w)=(k_1,\ldots,k_n),\quad \mbox{and}\quad \bm{m}(w)=(m_1,\ldots,m_\l).\]
		If $\pi_j=\partial_j x_j$ is the $j$th Demazure operator and $\omega_j$ denotes the fundamental weight $\omega_j=x_1x_2\cdots x_j$, then
		\[\mathfrak{S}_w = \omega_1^{k_1}\cdots\omega_n^{k_n}\pi_{i_1}(\omega_{i_1}^{m_1} \pi_{i_2}(\omega_{i_2}^{m_2}\cdots \pi_{i_\l}(\omega_{i_\l}^{m_\l})\cdots )). \]
	\end{theorem}
	
	\begin{example}
		For $w=31542$, it is easily checked that
		\begin{align*}
			\mathfrak{S}_w&=x_1\pi_2\pi_3(x_1x_2x_3\pi_1(x_1))\\
			&=x_2 x_3 x_1^3+x_2 x_4 x_1^3+x_3 x_4 x_1^3+x_2 x_3^2 x_1^2+x_2^2 x_3 x_1^2+x_2^2 x_4 x_1^2+x_3^2 x_4 x_1^2+x_2 x_3 x_4 x_1^2.
		\end{align*}
	\end{example}
	
	\section{Sorted Permutations and Grothendieck Polynomials}
	\label{sec:sortpermgroth}
	In this section, we define a special class of permutations, called sorted permutations. We introduce a projection map onto this class called $\sort$. We then relate the Grothendieck polynomials of any permutation and its image under $\sort$.
	
	\begin{definition}
		A \emph{standard interval} is a set of the form $[j]$ for some $j\geq 0$.
	\end{definition}

	Recall (see for instance \cite{manivel}) that a permutation $w$ is called \emph{dominant} if it satisfies any of the following equivalent conditions.
	\begin{itemize}
		\item There are no indices $i<j<k$ with $w(i)<w(k)<w(j)$ (called \emph{132-patterns}).
		\item The Rothe diagram $D(w)$ is the Young diagram of a partition.
		\item All columns of $D(w)$ are standard intervals.
	\end{itemize}
	
	\begin{definition}
		Fix a permutation $w \in S_n$. We define quantities $(h,C,\alpha,i_1,\beta)$ associated to $w$, collectively called the \emph{primary column data} of $w$. Assume first that $w$ is not dominant, so that $D(w)$ has a column which is not a standard interval. Let $h$ be the smallest integer such that the column $D(w)_{h+1}$ is not a standard interval. Denote by $C$ the column $C=D(w)_{h+1}\subseteq [n]$. Define $\alpha$ to be the largest integer such that $[\alpha]\subseteq C$. Denote by $i_1$ the smallest missing tooth of $C$. Lastly, set $\beta=i_1 - \alpha$, the size of the ``uppermost gap'' of $C$. If $w$ is dominant, simply set $h=n$, $C=\emptyset$, $\alpha=0$, $i_1=n$, and $\beta=n$. 
	\end{definition} 

	\begin{example}
		The permutation $w = 68432751$ has diagram shown in Figure \ref{fig:columndata}. The leftmost column that is not a standard interval is column five, so $h=4$ and $C=D(w)_5=\{1,2,6\}$. From $C$, we read off $\alpha=2$, $i_1=5$, and $\beta=3$.
	\end{example}

	\begin{figure}[ht]
		\begin{center}
			\includegraphics[scale=.85]{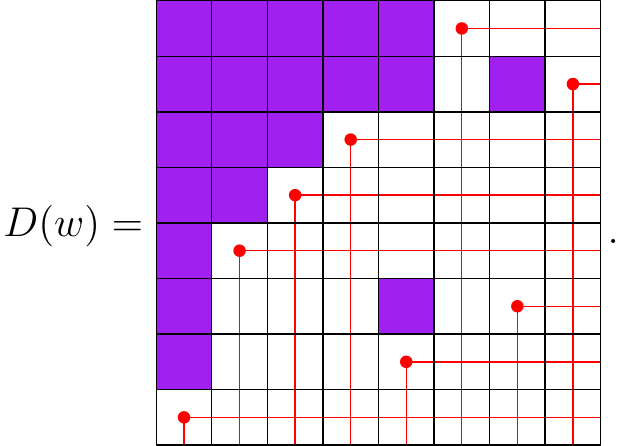}
		\end{center}
		\caption{The diagram of $w=68432751$.}
		\label{fig:columndata}
	\end{figure}

	\begin{lemma}
		\label{lem:bij}
		Any permutation $w$ restricts to a bijection
		\[[\alpha+1,i_1] \to [h-\beta+1,h ]. \]
		Moreover, the corresponding permutation $\sigma\in S_\beta$ is dominant.
	\end{lemma}
	\begin{proof}
		If $w$ is dominant, then $\sigma=w$ and there is nothing to prove. Assume $w$ is not dominant. For any $k \in [\alpha + 1,i_1]$, we have $w(k) \leq h$ since the $(h+1)$-th column $C$ of $D(w)$ has no box in row $k$, but has a box in row $i_1 + 1 > k$.		
		
		Suppose there exists $k \in[\alpha + 1,i_1]$ with $w(k) < h-\beta + 1$. Then we can find $p \in [h - \beta + 1, h]$ with $p\notin  w\left([\alpha + 1,i_1]\right)$. Consider $w^{-1}(p)$. By assumption, $w^{-1}(p) \not \in \{\alpha + 1, \ldots, i_1\}$. Since $[\alpha]\subseteq C$, all columns left of $C$ also contain $[\alpha]$. In particular, $[\alpha]\subseteq D(w)_p$ so $w^{-1}(p)\notin [\alpha]$. Since $i_1+1\in C$, $w^{-1}(p) \neq i_1 + 1$. Thus $w^{-1}(p) > i_1+1$. Since $i_1+1\in C$, this implies $i_1+1\in D(w)_p$.
		
		As $w(k)<h-\beta+1\leq p$, it must be that $k\notin D(w)_p$. However, this implies that $D(w)_p$ is not a standard interval, a contradiction to the definition of $h$. The assertion that the induced permutation $\sigma\in S_\beta$ is dominant follows easily from the fact that all columns left of $C$ are standard intervals containing $[\alpha]$.
	\end{proof}
	
	\begin{definition}
		Given $w \in S_n$, define $\sigma(w)\in S_\beta$ to be the dominant permutation obtained by restricting $w$ to $[\alpha + 1, i_1]$. We say $w$ is \emph{sorted} if $\sigma(w)$ is the identity permutation. The \emph{sorting} of $w$, denoted $w_\sort$, is the permutation obtained from $w$ by reordering the numbers $w(\alpha+1), \ldots, w(i_1)$ to be in increasing order.
	\end{definition}
	
	\begin{example}
		The permutation $w = 68432751$ has $\sigma(w)=321$. This implies $w_\sort = 68234751$. The diagrams of $w$ and $w_\sort$ are 
%		shown in Figure \ref{fig:diagramandsort}.
		\begin{center}
			\includegraphics[scale=.85]{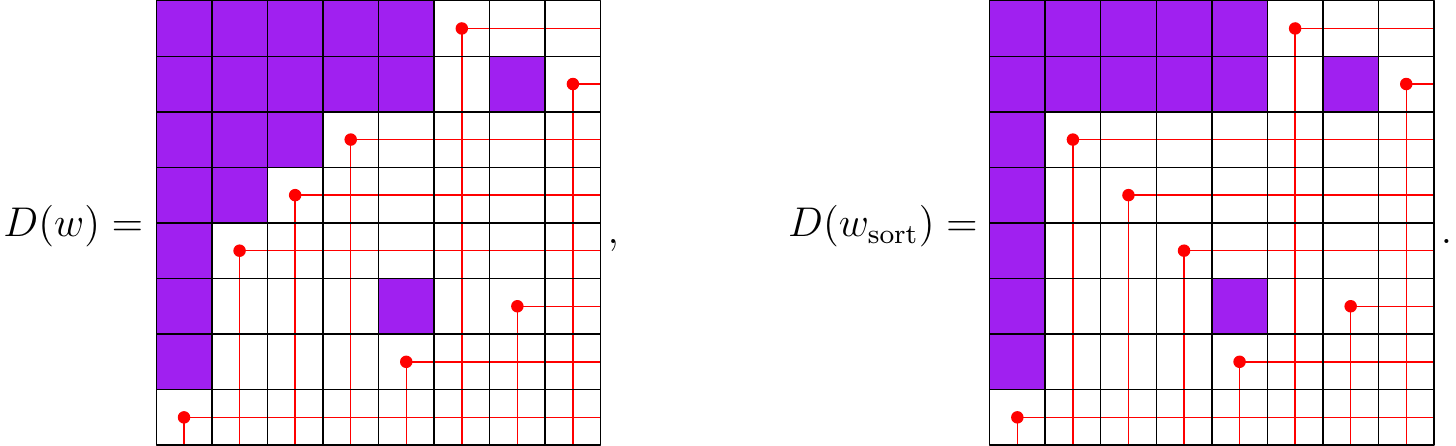}
		\end{center}
	\end{example}

%	\begin{figure}[ht]
%		
%	\caption{The diagrams of $w = 68432751$ and $w_{\sort}=68234751$.}
%	\label{fig:diagramandsort}
%	\end{figure}
	Note that any permutation $w_{\sort}$ is always sorted, and that the primary column data of $w_\sort$ is always the same as the primary column data of $w$. Observe also that $w_{\sort}$ is the identity permutation whenever $w$ is dominant.
	
	We now describe the relation between Grothendieck polynomials and the sort operation on permutations. We will write $(ab)$ for the transposition in $S_n$ swapping $a$ and $b$. We first recall a formula due to Lenart for the multiplication of a Grothendieck polynomial by a variable.
	
	Let $w\in S_n$ and $j\in[n]$. Denote by $(ab)$ the transposition in $S_n$ swapping the numbers $a$ and $b$. For $j\in[n]$, define the set $P_j(w)$ to consist of all permutations 
	\[v=w\cdot (a_1j)\cdots(a_pj)(jb_1)\cdots(jb_q)\in S_n \]
	such that $p,q\geq 0$, $p+q\geq 1$, 
	\[a_p<a_{p-1}<\cdots<j<b_q<b_{q-1}<\cdots<b_1, \]
	and the length increases by exactly 1 upon multiplication by each transposition. For $v\in P_j(w)$, define $\xi_j(w,v)=(-1)^{q+1}$.
	\begin{theorem}[{\cite[Theorem 3.1]{grothmonks}}]
		\label{thm:grothmonk}
		For any $w\in S_n$ and $j\in [n]$, 
		\[x_j\mathfrak{G}_w=\sum_{v\in P_j(w)}\xi_j(w,v)\mathfrak{G}_v. \]
	\end{theorem}
	
	Denote by $\mathrm{des}(w)$ the \emph{descent set} of $w$, $\mathrm{des}(w)=\{j\mid w(j)>w(j+1) \}$.
	\begin{lemma}
		\label{lem:schubertunsort}
		Let $w\in S_n$ be a nonidentity permutation with primary column data $(h,C,\alpha,i_1,\beta)$. Set
		\[a=\max \left(\mathrm{des}(w)\cap [\alpha+1,i_1]\right) \quad \mbox{and} \quad b = \max\left(\{p\mid w(p)<w(a)\}\cap [\alpha+1,i_1] \right).\]
		Then $\mathfrak{G}_w=x_a\mathfrak{G}_{w\cdot(ab)}$.
%		\[\mathfrak{G}_w=x_a\mathfrak{G}_{w\cdot(ab)}. \]
	\end{lemma}
	\begin{proof}
		The lemma can be proved by a straightforward but lengthy case analysis, using Theorem \ref{thm:grothmonk} to show $P_{a}(w\cdot (ab))=\{w\}$. Alternatively, observe that the particular choice of $a$ and $b$ implies $D(\sigma(w\cdot (ab)))$ equals $D(\sigma(w))$ with the rightmost box in row $a$ (which will be the bottommost row) removed. The box removed in $D(w)$ is bottomost in its column and rightmost in the dominant part of $D(w)$. The reader familiar with pipe dreams may note that this lemma is now a trivial consequence of the simplicial complex perspective of \cite{multidegree}, together with the ladder moves of \cite{laddermoves}.
	\end{proof}

	\begin{proposition}
		\label{prop:grothunsort}
		Let $w\in S_n$ and suppose $\sigma(w)$ has Rothe diagram equal to the Young diagram of $\lambda=(\lambda_1,\ldots,\lambda_{\beta})$. Then
		\[\mathfrak{G}_w=x_{\alpha+1}^{\lambda_1}\cdots x_{i_1}^{\lambda_{\beta}}\mathfrak{G}_{w_{\sort}}. \]
	\end{proposition}
	\begin{proof}
		It is enough to work inductively and use that the choice of $a$ and $b$ in Lemma \ref{lem:schubertunsort} implies $\sigma(w\cdot (ab))$ equals $\sigma(w)\cdot(a-\alpha\,\,b-\alpha)$, which has one fewer inversion than $\sigma(w)$.
	\end{proof}
	
	From Proposition \ref{prop:grothunsort}, one can immediately deduce the following well-known property of dominant permutations.
	\begin{corollary}
		If $w\in S_n$ is dominant with $D(w)$ equal to the Young diagram of $\lambda$, then
		\[\mathfrak{G}_w=x_1^{\lambda_1}\cdots x_{n-1}^{\lambda_{n-1}}. \]
	\end{corollary}

	\section{Orthodontia And Weak Bruhat Order}
	\label{sec:orthobruhat}
	In this section, we track how orthodontic sequences of permutations change with the application of certain adjacent transpositions and the sort operation. The following technical theorem deals with the case of sorted permutations, from which we move up in weak Bruhat order. Proposition \ref{prop:unsortedorthodontia} handles the unsorted case, in which we move down in weak Bruhat order. We offer an example first to help illustrate the sorted case and its proof.
	\begin{example}
		Consider $w = 68432751$ with $w_\sort = 68234751$. Recall the primary column data of $w_\sort$ is $h=4$, $C=\{1,2,6\}$, $\alpha=2$, $i_1=5$, and $\beta=3$. The diagrams of $w_{\sort}$, $w_{\sort}s_5$, $w_{\sort}s_5s_4$, and $w_{\sort}s_5s_4s_3$ are shown in Figure \ref{fig:inversesort}. The orthodontic sequence of $w_\sort$ is
		\[\bm{i}(w_\sort) = (5,4,3,1),\quad \bm{k}(w_\sort)=(0,3,0,0,0,0,1),\quad \bm{m}(w_\sort)=(0,0,1,1), \]
		and the orthodontic sequence of $w_\sort s_5s_4s_3$ is
		\[\bm{i}(w_\sort s_5s_4s_3) = (1),\quad \bm{k}(w_\sort s_5s_4s_3)=(0,0,4,0,0,0,1),\quad \bm{m}(w_\sort s_5s_4s_3)=(1). \]
	\end{example}

	\begin{figure}[ht]
		\begin{center}
			\includegraphics[scale=.85]{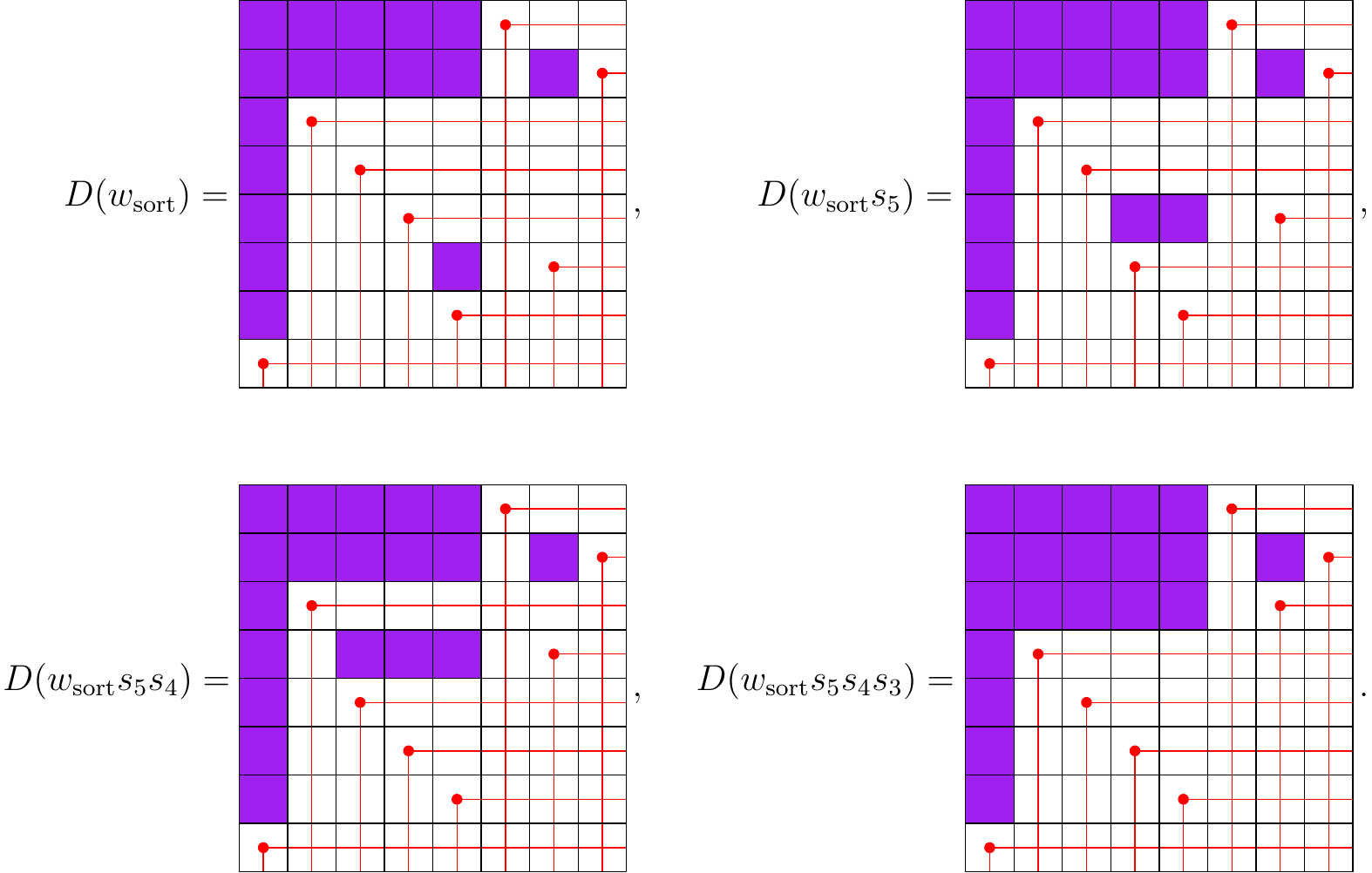}
		\end{center}
		\caption{The diagrams of $w_{\sort}$, $w_{\sort}s_5$, $w_{\sort}s_5s_4$, and $w_{\sort}s_5s_4s_3$ when $w=68432751$.}
		\label{fig:inversesort}
	\end{figure}

	\begin{theorem}
		\label{thm:sortedorthodontia}
		Let $w\in S_n$ be a nonidentity sorted permutation, and suppose $w$ has orthodontic sequence
		\[\bm{i}(w) = (i_1,\ldots, i_\l),\quad \bm{k}(w)=(k_1, \ldots, k_n), \mbox{ and}\quad \bm{m}(w) = (m_1, \ldots, m_\l).\] 
		Let $(h,C,\alpha,i_1,\beta)$ be the primary column data of $w$. Then:
		\begin{enumerate}[label=\textup{(\roman*)}]
			\item For $j \in [\beta]$, we have $i_j = i_1 - j + 1$.
			
			\item If $\alpha>0$, then $k_\alpha \geq \beta$.
			
			\item For $j \in [\alpha + 1, i_1]$, we have $k_j = 0$. 
			
			\item For $j \in [\beta - 1]$, we have $m_j = 0$.
			
			\item The permutation $ws_{i_1}\cdots s_{\alpha + 1}$ has orthodontic sequence
			\[\bm{i}(ws_{i_1}\cdots s_{\alpha+1}) = (i_{\beta + 1}, \ldots, i_\l),\quad 
			\bm{m}(ws_{i_1}\cdots s_{\alpha+1})=(m_{\beta + 1}, \ldots, m_\l),				
			\]
			and
			\begin{align*}
			\bm{k}(ws_{i_1}\cdots s_{\alpha+1}) = \begin{cases}
			(k_1, \ldots,k_{\alpha-1}, k_\alpha - \beta, \beta + m_\beta, k_{\alpha+2},\ldots, k_n) &\mbox{ if $\alpha>0$},\\
			(\beta + m_\beta, k_{2},\ldots, k_n) &\mbox{ if $\alpha=0$}.\\
			\end{cases}	
			\end{align*}
		\end{enumerate}
	\end{theorem}
	\begin{proof}
		By definition, $C$ contains $[\alpha] \cup\{i_1 + 1\}$ and does not contain any of $\alpha + 1, \ldots, i_1$. It follows that the orthodontic sequence begins $(i_1, i_1 - 1, \ldots, \alpha+1)$. This proves $\textup{(i)}$. To prove $\textup{(ii)}$, observe that since $w$ is sorted, the $\beta$ columns immediately left of $C$ are all equal to $[\alpha]$. Thus if $\alpha>0$, then $k_\alpha \geq \beta$. 
		
		For $\textup{(iii)}$, suppose there is a column $C'$ equal to $[j]$ for $j\in[\alpha+1,i_1]$. Let $C'$ be column $p$ of $D(w)$. Since $[\alpha]\subseteq C$ but $\alpha+1\notin C$, it follows that $p\leq h$. Since $w$ is sorted, the $\beta$ columns left of $C$ all equal $\alpha$. Thus $p\leq h-\beta$. Consider $w(j+1)$. Since $j+1\notin C'$, $w(j+1)\leq p\leq h-\beta$. As $w([\alpha+1,i_1])=[h-\beta+1,h]$ and $j+1\in[\alpha+2,i_1+1]$, it follows that $j+1=i_i+1$. But $w(i_1+1)\leq h-\beta$ contradicts that $i_1+1\in C$, so there can be no such $C'$. 
		
		For $\textup{(iv)}$, consider the diagram $D(w)_-$ obtained by removing any standard intervals from $D(w)$. The first $\beta - 1$ steps of the orthodontia algorithm amount to permuting rows $(\alpha + 1, \ldots, i_1 + 1)$ of $D(w)_-$ to $(\alpha + 1, i_1 + 1, \alpha + 2, \ldots, i_1)$. Since $w$ is sorted, the $\beta$ many rows $\alpha +1, \ldots, i_1$ of $D(w)_-$ are all empty. Thus, $m_j = 0$ for $j\in[\beta-1]$. 
		
		Lastly, we prove $\textup{(v)}$. Consider the columns of $D(w)$. Columns $1,2\ldots,h-\beta$ are standard intervals that strictly contain $[\alpha]$. Since $w$ is sorted, columns $h-\beta+1,\ldots,\beta$ are each exactly $[\alpha]$. Let $E$ denote the diagram whose columns are the columns of $D(w)$ weakly to the right of $C$, with the same indices. Note that $E$ may contain standard intervals $[j]$ with $j\leq \alpha$, but $E$ has no boxes in rows $\alpha+1,\ldots,i_1$. 
		
		Let $\bm{k}(ws_{i_1}\cdots s_{\alpha+1}) = (k_1',\ldots,k_n')$. We analyze the columns of $D(ws_{i_1}\cdots s_{\alpha+1})$. Columns $1,2,\ldots,h-\beta$ of $D(ws_{i_1}\cdots s_{\alpha+1})$ agree with those of $D(w)$. Columns $h-\beta+1,\ldots,h$ are each $[\alpha+1]$. The remaining columns are exactly $s_{\alpha+1}\cdots s_{i_1}\cdot E$. 
		
		The only columns of $D(w)$ that can be of the form $[j]$ with $j\geq [\alpha+2]$ are columns $1,2,\ldots,h-\beta$. No new such columns are created by the action of $s_{\alpha+1}\cdots s_{i_1}$, so $k_j'=k_j$ for $j\geq\alpha+1$. Any standard intervals occurring in $E$ are weakly contained in $[\alpha]$, and so are unaffected by the action of $s_{\alpha+1}\cdots s_{i_1}$. Thus, $k_j'=k_j$ for $j\in [\alpha-1]$. 
		
		The $\beta$ columns $[\alpha]$ of $D(w)$ become $\beta$ copies of $[\alpha+1]$ in $D(ws_{i_1}\cdots s_{\alpha+1})$, but no other columns $[\alpha]$ are changed. Thus (if $\alpha>0$) $k_{\alpha}' = k_{\alpha}-\beta$. From $(\textup{iii})$, no columns $[\alpha+1]$ can occur in $D(w)$. However, $\beta$ more columns $[\alpha+1]$ appear in $D(ws_{i_1}\cdots s_{\alpha+1})$ from the $\beta$ columns $[\alpha]$ left of $C$ in $D(w)$. The $m_\beta$ columns in $E$ that were standardized to $[\alpha+1]$ by orthodontia all equal $[\alpha+1]$ as well in $D(ws_{i_1}\cdots s_{\alpha+1})$. Thus, we have $k_{\alpha+1}' = \beta+m_\beta$.
		
		The proof of $\textup{(v)}$ is completed by noting that columns $h+1,\ldots, n$ of $D(ws_{i_1}\cdots s_{\alpha+1})$ equal $s_{\alpha+1}\cdots s_{i_1}\cdot (E)$. This implies that $D(ws_{i_1}\cdots s_{\alpha+1})_-$ occurs in the execution of the orthodontia algorithm on $D(w)$ after $\beta$ steps. Hence 
		\[\bm{i}(ws_{i_1}\cdots s_{\alpha+1}) = (i_{\beta + 1}, \ldots, i_\l),\mbox{ and } 
		\bm{m}(ws_{i_1}\cdots s_{\alpha+1})=(m_{\beta + 1}, \ldots, m_\l).	\qedhere		
		\]
	\end{proof}
	
	Note that when $\alpha=0$, the analysis of $\bm{k}$ in the proof of part $\textup{(v)}$ above reduces to 
	\[
	\bm{k}(ws_{i_1}\cdots s_{\alpha+1}) = (\beta + m_\beta, k_{\alpha+2},\ldots, k_n).
	\]
	The following example illustrates this case.

	\begin{example}
		Consider the sorted permutation $w=12845376$ with primary column data $h=2$, $C=\{3,4,5\}$, $\alpha=0$, $i_1=2$, and $\beta=2$. The diagrams of $w$ and $ws_2s_1$ are
		\begin{center}
			\includegraphics[scale=.85]{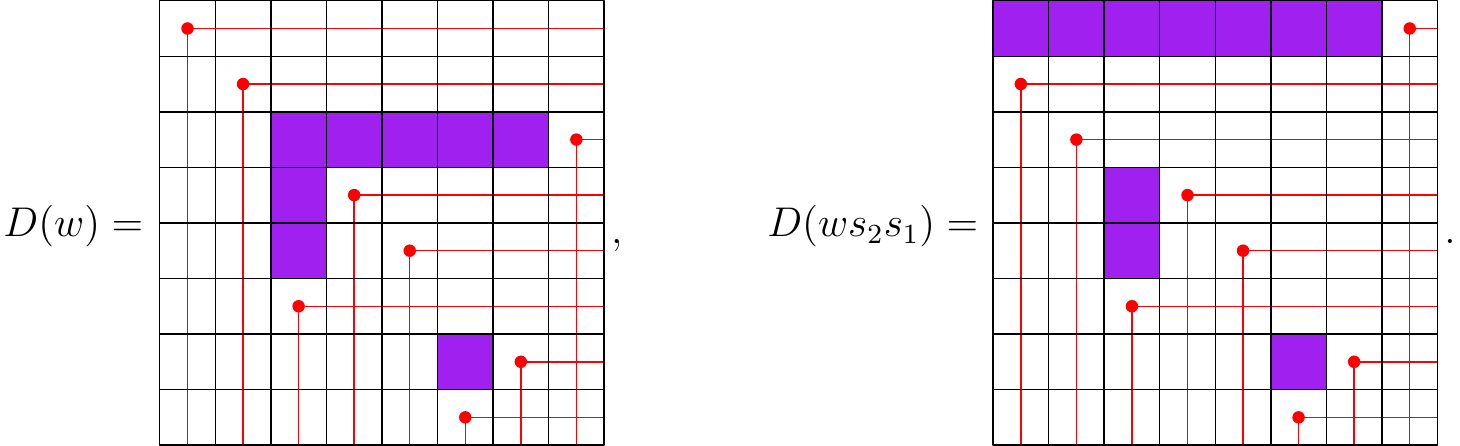}
		\end{center}
	The orthodontic sequence of $w$ is
	\[\bm{i}(w) = (2,1,3,2,4,3,6,5,4,3,2),\quad \bm{k}(w)=(0,0,0,0,0,0,0),\quad \bm{m}(w)=(0,3,0,0,0,1,0,0,0,0,1), \]
	and the orthodontic sequence of $ws_2s_1$ is
	\[\bm{i}(w s_2s_1) = (3,2,4,3,6,5,4,3,2),\quad \bm{k}(w s_5s_4s_3)=(5,0,0,0,0,0,0),\quad \bm{m}(w s_5s_4s_3)=(0,0,0,1,0,0,0,0,1). \]
	\end{example}
	
	We now connect the orthodontic sequence of any permutation $w$ to that of $w_\sort$. 
	
	\begin{proposition}
		\label{prop:unsortedorthodontia}
		Fix $w\in S_n$, and let $w$ have primary column data $(h,C,\alpha,i_1,\beta)$. Let $\sigma=\sigma(w)\in S_\beta$.
 		Suppose $w_\sort$ has orthodontic sequence 
		\[\bm{i}(w_\sort) = (i_1,\ldots, i_\l),\quad \bm{k}(w_\sort)=(k_1, \ldots, k_n), \mbox{ and}\quad \bm{m}(w_\sort) = (m_1, \ldots, m_\l).\] 
		Then $w$ has orthodontic sequence 
		\[\bm{i}(w) = \bm{i}(w_\sort), \quad \bm{m}(w)=\bm{m}(w_\sort), \mbox{ and} \quad \bm{k}(w) = (k_1',k_2',\ldots,k_n'), \]
		where
		\[k_j'=\begin{cases}
		k_j&\mbox{ if } j\leq \alpha-1,\\
		k_j-\bm{k}(\sigma)_1-\cdots-\bm{k}(\sigma)_{\beta}& \mbox{ if } j=\alpha,\\
		k_j+\bm{k}(\sigma)_{j-\alpha}&\mbox{ if } j\in[\alpha+1,i_1],\\
		k_j &\mbox{ if } j\geq i_1+1.
		\end{cases} \]		
	\end{proposition}
	\begin{proof}
		This result follows easily from the fact that $D(w_\sort)$ is obtained from $D(w)$ by removing $D(\sigma)$ from the square $[\alpha+1,i_1]\times [h-\beta+1,h]$ inside $D(w)$. The only columns of $D(w)$ affected by this are standard intervals, and they stay standard intervals after the removal.
	\end{proof}

	\section{An Orthodontia Formula for Grothendieck Polynomials}
	\label{sec:grothortho}
	In this section we extend the orthodontic operator formula from Schubert polynomials to Grothendieck polynomials by replacing Demazure operators $\pi_j$ by Demazure--Lascoux operators $\overline{\pi}_j$. We first construct a partial order on $S_n$, which we will induct over to prove the extension. 
	\begin{definition}
		For $w\in S_n$, define $\fb(w)$ to be the set of \emph{fallen boxes} of $w$, the pairs $(i,j)\in D(w)$ that are not top-aligned in $D(w)$.
	\end{definition}
	
	\begin{example}
		\label{exp:132patterns}
		The sorted permutation $w=58134726$ has diagram 
		\begin{center}
			\includegraphics[scale=.85]{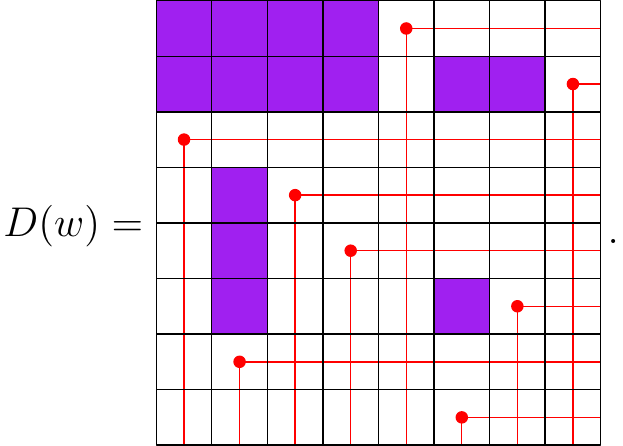}
		\end{center}
		Thus, $\fb(w)=\{(2,6),(2,7),(4,2),(5,2),(6,2),(6,6) \}$.
	\end{example}
	
	\begin{lemma}
		\label{lem:fb-less-equal}
		For any permutation $w\in S_n$, 
		\[\fb(w)=\fb(w_\sort). \]
		If $w$ is nonidentity and sorted with primary column data $(h,C,\alpha,i_1,\beta)$, then 
		\[\#\fb(ws_{i_1}\cdots s_{\alpha+1})<\#\fb(w).\]
	\end{lemma}
	\begin{proof}
		The claim $\fb(w)=\fb(w_\sort)$ follows from fact that $D(w_\sort)$ is obtained from $D(w)$ by removing any boxes lying in $[\alpha+1,i_1]\times [h-\beta+1,h]$. The boxes removed constitute a bottommost and rightmost subset of the dominant part of $w$.
		
		For $w$ nonidentity and sorted, the assertion $\#\fb(ws_{i_1}\cdots s_{\alpha+1})<\#\fb(w)$ follows from Theorem \ref{thm:sortedorthodontia} parts $\textup{(i)}$, $\textup{(iv)}$, and $\textup{(v)}$.		
	\end{proof}
	
	\begin{example}
		In Example \ref{exp:132patterns}, we saw the sorted permutation $w=23854716$ had $\#\fb(w)= 6$. The primary column data of $w$ is $h=1$, $C=\{1,2,4,5,6\}$, $\alpha=2$, $i_1=3$, and $\beta=1$, so $ws_{i_1}\cdots s_{\alpha+1}=ws_3=58314726$. Then
		\begin{center}
			\includegraphics[scale=.85]{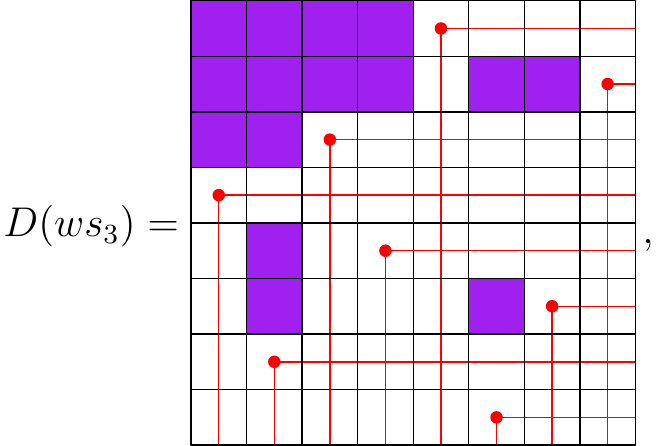}
		\end{center}
		so $\#\fb(ws_{3}) = \#\{(2,6), (2,7), (5,2), (6,2), (6,6)\} = 5$.
	\end{example}
	
	\begin{definition}
		Suppose $w\in S_n$ has primary column data $(h,C,\alpha,i_1,\beta)$. Define the \emph{orthodontic sort order} $\leq_{\mathrm{os}}$ on $S_n$ as the reflexive and transitive closure of the relations
		\begin{align*}
		\label{eqn:daggerx1}
		\tag{$\dagger$} w_\sort &\preccurlyeq w, \mbox{ and}\\
		\label{eqn:daggerx2}
		\tag{$\ddagger$}ws_{i_1}\cdots s_{\alpha+1} &\preccurlyeq w \mbox{ whenever $w$ is nonidentity and sorted}.
		\end{align*}
	\end{definition}
	
	\begin{proposition}
		\label{prop:partialorder}
		The relation $\leq_{\mathrm{os}}$ is a partial order on $S_n$, and the identity is the minimum element.
	\end{proposition}
	\begin{proof}
		Reflexivity and transitivity follow immediately from the definition. It remains to show antisymmetry. Assume we have $u,v\in S_n$ with $u\leq_{\mathrm{os}} v$ and $v\leq_{\mathrm{os}} u$. Then there are chains
		\begin{align*}
		u&=w_1\preccurlyeq w_2\preccurlyeq\cdots\preccurlyeq w_k=v,\mbox{ and }\\
		v&=w_1'\preccurlyeq w_2'\preccurlyeq\cdots\preccurlyeq w_m'=u.
		\end{align*}
		By Lemma \ref{lem:fb-less-equal}, applying $\#\fb$ to both chains yields
		\[\#\fb(u)\leq\#\fb(w_1)\leq \cdots \leq \#\fb(v)\mbox{ and } \#\fb(v)\leq \#\fb(w'_1)\leq\cdots\leq \#\fb(u). \]
		Thus, $\#\fb(u)=\#\fb(v)$, so the function $\#\fb(\cdot)$ is constant on both $\preccurlyeq$ chains. Consequently, all relations appearing in either chain of $\preccurlyeq$'s must be of type $(\dagger)$. This easily implies $u=v$, since $w_{\sort}=w$ whenever $w$ is sorted. The fact that the identity permutation is the minimum follows from an analogous argument.
	\end{proof}
	
	Recall the Demazure--Lascoux operators $\overline{\pi}_j$, defined by 
	\[\overline{\pi}_j(f) = \partial_j(x_j(1-x_{j+1})f). \] We use $\overline{\pi}_j$ to define an orthodontia polynomial $\mathscr{G}_w$.
	
	\begin{definition}
		Pick any $w\in S_n$, and suppose $w$ has orthodontic sequence
		\[\bm{i}(w) = (i_1,\ldots, i_\l),\quad \bm{k}(w)=(k_1, \ldots, k_n), \mbox{ and}\quad \bm{m}(w) = (m_1, \ldots, m_\l).\] 
		Define a polynomial $\mathscr{G}_w$ by 
		\[
		\mathscr{G}_w= \omega_1^{k_1}\cdots\omega_n^{k_n}\overline{\pi}_{i_1}(\omega_{i_1}^{m_1}\overline{\pi}_{i_2}(\omega_{i_2}^{m_2}\cdots\overline{\pi}_{i_\l}(\omega_{i_\l}^{m_\l})\cdots)),
		\]
		where $\omega_i$ is the fundamental weight $\omega_i=x_1\cdots x_i$.
	\end{definition}

	The following five lemmas form the technical heart of the proof that $\mathscr{G}_w=\mathfrak{G}_w$ (Theorem \ref{thm:groth-equality}).
	
	\begin{lemma}
		\label{lem:computational-groth1}
		Let $g\in\mathbb{C}[x_1,\ldots,x_n]$ be any polynomial, and fix $j\in[n-1]$. For any $\delta\geq 1$,
		\begin{align*}
		\overline{\partial}_{j}(x_{j}^\delta g)
		&=\partial_{j}(g)(x_{j+1}^{\delta}-x_{j}x_{j+1}^{\delta})+
		g\left(\sum_{q=0}^{\delta-1}x_{j}^qx_{j+1}^{\delta-1-q}\right)-g\left(\sum_{q=0}^{\delta-2}x_{j}^{q+1}x_{j+1}^{\delta-1-q}\right).
		\end{align*}
	\end{lemma}
	\begin{proof}
		If $\delta=1$, we have
		\[\overline{\partial}_{j}(x_{j}g) =\partial_{j}\left((1-x_{j+1})x_{j}g\right) = \partial_{j}\left(x_{j}g\right)-\partial_j\left(x_{j}x_{j+1}g\right). \]
		It is easy to check
		\[\partial_j(x_jg)=x_{j+1}\partial_{j}(g)+g\quad\mbox{and}\quad\partial_j(x_jx_{j+1}g)=x_jx_{j+1}\partial_j(g).\]
		Thus,
			\[\overline{\partial}_{j}(x_{j}g) = x_{j+1}\partial_{j}(g)+g-x_jx_{j+1}\partial_j(g) = \partial_{j}(g)(x_{j+1}-x_jx_{j+1})+g. \]
		For $\delta>1$, expand out $\overline{\partial}_{j}$ to get
		\begin{align*}
			\overline{\partial}_{j}(x_{j}^\delta g)=\partial_{j}(x_{j}^\delta g) - \partial_{j}(x_{j}^{\delta}x_{j+1}g).
		\end{align*}
		Straightforward computations show that
		\begin{align}
			\label{eqn:1-groth}
			\partial_{j}(x_{j}^{\delta}g) 
			&=x_{j+1}^{\delta}\partial_{j}(g)+g\sum_{q=0}^{\delta-1}x_{j}^qx_{j+1}^{\delta-1-q}
		\end{align}
		and
		\begin{align*}
			\partial_{j}(x_{j}^{\delta}x_{j+1}g) 
			&= gx_{j}^{\delta-1}x_{j+1}+x_{j}x_{j+1}^2\partial_{j}(x_{j}^{\delta-2}g).
		\end{align*}
		Thus,
		\begin{align*}
			\overline{\partial}_{j}(x_{j}^\delta g)&=
			x_{j+1}^{\delta}\partial_{j}(g)+g\left(\sum_{q=0}^{\delta-1}x_{j}^qx_{j+1}^{\delta-1-q}\right)-
			gx_{j}^{\delta-1}x_{j+1}-x_{j}x_{j+1}^2\partial_{j}(x_{j}^{\delta-2}g).
		\end{align*}
		Using (\ref{eqn:1-groth}), we can expand $\partial_j(x_j^{\delta-2}g)$ as
		\begin{align*}
			\partial_{j}(x_{j}^{\delta-2}g) 
			&=x_{j+1}^{\delta-2}\partial_{j}(g)+g\sum_{q=0}^{\delta-3}x_{j}^qx_{j+1}^{\delta-3-q}.
		\end{align*}
		This implies 
		\begin{align*}
			\overline{\partial}_{j}(x_{j}^\delta g)
			&=x_{j+1}^{\delta}\partial_{j}(g)+g\left(\sum_{q=0}^{\delta-1}x_{j}^qx_{j+1}^{\delta-1-q}\right)-
			gx_{j}^{\delta-1}x_{j+1}-x_{j}x_{j+1}^2\left(x_{j+1}^{\delta-2}\partial_{j}(g)+g\sum_{q=0}^{\delta-3}x_{j}^qx_{j+1}^{\delta-3-q}\right)\\
			&=x_{j+1}^{\delta}\partial_{j}(g)+g\left(\sum_{q=0}^{\delta-1}x_{j}^qx_{j+1}^{\delta-1-q}\right)-
			gx_{j}^{\delta-1}x_{j+1}-\left(x_jx_{j+1}^{\delta}\partial_{j}(g)+g\sum_{q=0}^{\delta-3}x_{j}^{q+1}x_{j+1}^{\delta-1-q}\right)\\
			&=\partial_{j}(g)(x_{j+1}^{\delta}-x_jx_{j+1}^{\delta})+g\left(\sum_{q=0}^{\delta-1}x_{j}^qx_{j+1}^{\delta-1-q}\right)-
			gx_{j}^{\delta-1}x_{j+1}-g\left(\sum_{q=0}^{\delta-3}x_{j}^{q+1}x_{j+1}^{\delta-1-q}\right)\\
			&=\partial_{j}(g)(x_{j+1}^{\delta}-x_jx_{j+1}^{\delta})+g\left(\sum_{q=0}^{\delta-1}x_{j}^qx_{j+1}^{\delta-1-q}\right)-g\left(\sum_{q=0}^{\delta-2}x_{j}^{q+1}x_{j+1}^{\delta-1-q}\right).\qedhere
		\end{align*}
		
	\end{proof}

	\begin{lemma}
		\label{lem:computational-groth2}
		Let $g$ be a polynomial with 
		\[\overline{\partial}_{j+1}(g) = \cdots = \overline{\partial}_{j+\gamma - 1}(g)= g\quad\mbox{and}\quad \partial_{j+1}(g) = \cdots = \partial_{j+\gamma - 1}(g)= 0\]
		for some $\gamma\geq 2$. Then for any $\delta\geq 0$,
		\begin{align*}
		\overline{\partial}_{j+\gamma-1}\cdots \overline{\partial}_{j+1}\left(x_{j+1}^\delta g\right)=g.
		\end{align*}
	\end{lemma}
	\begin{proof}
		We work by induction on $\delta$. The base case $\delta=0$ follows from the assumptions on $g$. Assume the result holds for all $\delta'<\delta$. From Lemma \ref{lem:computational-groth1}, we obtain
		\begin{align*}
			\overline{\partial}_{j+1}(x_{j+1}^\delta g)
			&=\partial_{j+1}(g)(x_{j+2}^{\delta}-x_{j+1}x_{j+2}^{\delta})+
			g\left(\sum_{q=0}^{\delta-1}x_{j+1}^qx_{j+2}^{\delta-1-q}\right)-g\left(\sum_{q=0}^{\delta-2}x_{j+1}^{q+1}x_{j+2}^{\delta-1-q}\right)\\
			&=g\left(\sum_{q=0}^{\delta-1}x_{j+1}^qx_{j+2}^{\delta-1-q}\right)-g\left(\sum_{q=0}^{\delta-2}x_{j+1}^{q+1}x_{j+2}^{\delta-1-q}\right)
		\end{align*}
		since $\partial_{j+1}(g)=0$ by assumption.
		Plugging into $\overline{\partial}_{j+\gamma-1}\cdots \overline{\partial}_{j+1}\left(x_{j+1}^\delta g\right)$ yields
		\begin{align*}
			\overline{\partial}_{j+\gamma-1}\cdots \overline{\partial}_{j+1}\left(x_{j+1}^\delta g\right)
			&=\overline{\partial}_{j+\gamma-1}\cdots \overline{\partial}_{j+2}
			\left(\sum_{q=0}^{\delta-1}x_{j+1}^qx_{j+2}^{\delta-1-q}g\right)-
			\overline{\partial}_{j+\gamma-1}\cdots \overline{\partial}_{j+2}\left(\sum_{q=0}^{\delta-2}x_{j+1}^{q+1}x_{j+2}^{\delta-1-q}g\right)\\
			&=\sum_{q=0}^{\delta-1}\overline{\partial}_{j+\gamma-1}\cdots \overline{\partial}_{j+2}\left(x_{j+1}^qx_{j+2}^{\delta-1-q}g\right)-
			\sum_{q=0}^{\delta-2}\overline{\partial}_{j+\gamma-1}\cdots \overline{\partial}_{j+2}\left(x_{j+1}^{q+1}x_{j+2}^{\delta-1-q}g\right)\\
			&=\sum_{q=0}^{\delta-1}x_{j+1}^q\overline{\partial}_{j+\gamma-1}\cdots \overline{\partial}_{j+2}\left(x_{j+2}^{\delta-1-q}g\right)-
			\sum_{q=0}^{\delta-2}x_{j+1}^{q+1}\overline{\partial}_{j+\gamma-1}\cdots \overline{\partial}_{j+2}\left(x_{j+2}^{\delta-1-q}g\right).\\
		\end{align*}
		Applying the induction assumption to each summand gives
		\begin{align*}
			\overline{\partial}_{j+\gamma-1}\cdots \overline{\partial}_{j+1}\left(x_{j+1}^\delta g\right)
			&=\sum_{q=0}^{\delta-1}x_{j+1}^qg-
			\sum_{q=0}^{\delta-2}x_{j+1}^{q+1}g = g. \qedhere
		\end{align*}
	\end{proof}

	\begin{lemma}
		\label{lem:computational-groth3}
		Let $g$ be a polynomial with 
		\[\overline{\partial}_{j+1}(g) = \cdots = \overline{\partial}_{j+\gamma - 1}(g)= g\quad\mbox{and}\quad \partial_{j+1}(g) = \cdots = \partial_{j+\gamma - 1}(g)= 0\]
		for some $\gamma\geq 2$. Set $g'=\overline{\pi}_j(g)$, $j'=j+1$, and $\gamma'=\gamma-1$. Then,
		\[\overline{\partial}_{j'+1}(g') = \cdots = \overline{\partial}_{j'+\gamma' - 1}(g')= g'\quad\mbox{and}\quad \partial_{j'+1}(g') = \cdots = \partial_{j'+\gamma' - 1}(g')= 0.\]
	\end{lemma}
	\begin{proof}
		We need to show
		\[\overline{\partial}_{j+2}(\overline{\pi}_j(g)) = \cdots = \overline{\partial}_{j+\gamma-1}(\overline{\pi}_j(g))= \overline{\pi}_j(g)\quad\mbox{and}\quad \partial_{j+2}(\overline{\pi}_j(g)) = \cdots = \partial_{j+\gamma-1}(\overline{\pi}_j(g))= 0.\]
		Let $k\in [j+2,j+\gamma-1]$. Since $|k-j|>1$, it follows that $\overline{\partial}_k\overline{\partial}_j=\overline{\partial}_j\overline{\partial}_k$. This yields
		\[
			\overline{\partial}_{k}(\overline{\pi}_j(g))= 
			\overline{\partial}_{k}(\overline{\partial}_j(x_jg))=
			\overline{\partial}_{j}(\overline{\partial}_k(x_jg))=
			\overline{\partial}_{j}(x_j\overline{\partial}_k(g))=
			\overline{\pi}_{j}(\overline{\partial}_k(g))=
			\overline{\pi}_{j}(g).
		\]
		By identical argument replacing $\overline{\partial}_k$ by $\partial_k$, one obtains $\partial_{k}(\overline{\pi}_j(g))=0$.
	\end{proof}

	\begin{lemma}
		\label{lem:computational-groth4}
		Let $g$ be a polynomial with 
		\[\overline{\partial}_{j+1}(g) = \cdots = \overline{\partial}_{j+\gamma - 1}(g)= g\quad\mbox{and}\quad \partial_{j+1}(g) = \cdots = \partial_{j+\gamma - 1}(g)= 0\]
		for some $\gamma\geq 2$. Then,
		\begin{align*}
		\overline{\partial}_{j+\gamma-1}\cdots\overline{\partial}_j\left(x_j^\gamma g\right) = \overline{\pi}_{j+\gamma-1}\cdots\overline{\pi}_j\left(g\right).
		\end{align*}
	\end{lemma}
	\begin{proof}
		We work by induction on $\gamma$, with the base case $\gamma=1$ simply being the identity
		\[\overline{\partial}_{j}\left(x_jg\right)=\overline{\pi}_j\left(g\right).\] 
		From Lemma \ref{lem:computational-groth1}, we obtain 
		\begin{align*}
		\overline{\partial}_{j}(x_{j}^\gamma g)
		&=\partial_{j}(g)(x_{j+1}^{\gamma}-x_{j}x_{j+1}^{\gamma})+
		g\left(\sum_{q=0}^{\gamma-1}x_{j}^qx_{j+1}^{\gamma-1-q}\right)-g\left(\sum_{q=0}^{\gamma-2}x_{j}^{q+1}x_{j+1}^{\gamma-1-q}\right)
		\end{align*}
		By linearity and Lemma \ref{lem:computational-groth2},
		\begin{align}
			\overline{\partial}_{j+\gamma - 1}\cdots\overline{\partial}_j\left(x_j^\gamma g\right) 
			&=\overline{\partial}_{j+\gamma - 1}\cdots\overline{\partial}_{j+1}
			\left(\partial_{j}(g)(x_{j+1}^{\gamma}-x_{j}x_{j+1}^{\gamma})\right)+g\nonumber\\
			&=\overline{\partial}_{j+\gamma - 1}\cdots\overline{\partial}_{j+1}
			\left(\partial_{j}(g)(x_{j+1}^{\gamma}-x_{j}x_{j+1}^{\gamma})+g\right).
			\label{eqn:2-groth}
		\end{align}
		Note the second equality follows from the assumptions on $g$.
		
		On the other hand, Lemma \ref{lem:computational-groth1} implies
		\begin{align}
			\label{eqn:3-groth}
			x_{j+1}^{\gamma-1}\overline{\pi}_j(g)&=x_{j+1}^{\gamma-1}\overline{\partial}_j(x_jg)\nonumber\\
			&=\partial_{j}(g)(x_{j+1}-x_{j}x_{j+1})x_{j+1}^{\gamma-1}+gx_{j+1}^{\gamma-1}.
		\end{align}
		
		We claim that
		\begin{align*}
			\overline{\partial}_{j+\gamma - 1}\cdots\overline{\partial}_{j+1}
			\left(\overline{\partial}_j\left(x_j^\gamma g\right) - x_{j+1}^{\gamma-1}\overline{\pi}_j(g)\right)=0.
		\end{align*}
		Using (\ref{eqn:2-groth}) and (\ref{eqn:3-groth}), we compute
		\begin{align*}
			\overline{\partial}_{j+\gamma - 1}\cdots\overline{\partial}_{j+1}
			\left(\overline{\partial}_j\left(x_j^\gamma g\right) - x_{j+1}^{\gamma-1}\overline{\pi}_j(g)\right)&=
%			&=\overline{\partial}_{j+\gamma - 1}\cdots\overline{\partial}_{j+1}\left(\partial_{j}(g)(x_{j+1}^{\gamma}-x_{j}x_{j+1}^{\gamma})+g
%			-\left(\partial_{j}(g)(x_{j+1}-x_{j}x_{j+1})x_{j+1}^{\gamma-1}+gx_{j+1}^{\gamma-1}\right) \right)\\
%			&=\overline{\partial}_{j+\gamma - 1}\cdots\overline{\partial}_{j+1}\left(\partial_{j}(g)(x_{j+1}^{\gamma}-x_{j}x_{j+1}^{\gamma}-(x_{j+1}-x_{j}x_{j+1})x_{j+1}^{\gamma-1})+g-x_{j+1}^{\gamma-1}g \right)\\
			\overline{\partial}_{j+\gamma - 1}\cdots\overline{\partial}_{j+1}\left( g-x_{j+1}^{\gamma-1}g\right).\\
		\end{align*}
		Then by Lemma \ref{lem:computational-groth2} and the assumptions on $g$,
		\begin{align*}
			\overline{\partial}_{j+\gamma - 1}\cdots\overline{\partial}_{j+1}
			\left(\overline{\partial}_j\left(x_j^\gamma g\right) - x_{j+1}^{\gamma-1}\overline{\pi}_j(g)\right)
			&=\overline{\partial}_{j+\gamma - 1}\cdots\overline{\partial}_{j+1}\left(g-x_{j+1}^{\gamma-1}g\right)\\
			&=\overline{\partial}_{j+\gamma - 1}\cdots\overline{\partial}_{j+1}\left(g\right)-\overline{\partial}_{j+\gamma - 1}\cdots\overline{\partial}_{j+1}\left(x_{j+1}^{\gamma-1}g\right)\\
			&=g-g\\
			&=0.
		\end{align*}
		This establishes the claim, which implies
		\[\overline{\partial}_{j+\gamma - 1}\cdots\overline{\partial}_{j}
		\left(x_j^\gamma g\right)= \overline{\partial}_{j+\gamma - 1}\cdots\overline{\partial}_{j+1}\left( x_{j+1}^{\gamma-1}\overline{\pi}_j(g)\right). \]
		
		By Lemma \ref{lem:computational-groth3}, the polynomial $g'=\overline{\pi}_j(g)$ satsifies the assumptions of the induction hypothesis with $j'=j+1$ and $\gamma'=\gamma-1$. Then by induction,
		\begin{align*}
			\overline{\partial}_{j+\gamma - 1}\cdots\overline{\partial}_{j}
			\left(x_j^\gamma g\right)
			&=\overline{\partial}_{j+\gamma - 1}\cdots\overline{\partial}_{j+1}\left( x_{j+1}^{\gamma-1}\overline{\pi}_j(g)\right)\\
			&=\overline{\partial}_{j'+\gamma' - 1}\cdots\overline{\partial}_{j'+1}\left( x_{j+1}^{\gamma-1}g'\right)\\
			&=\overline{\pi}_{j'+\gamma' - 1}\cdots\overline{\pi}_{j'+1}\left(g'\right)\\
			&=\overline{\pi}_{j+\gamma - 1}\cdots\overline{\pi}_j\left(g\right).\qedhere
		\end{align*}
	\end{proof}
	
	\begin{lemma}
		\label{lem:computational-groth5}
		Let $w\in S_n$ be any sorted permutation. Suppose $w$ has primary column data $(h,C,\alpha,i_1,\beta)$. Set $g = x_{\alpha+1}^{-\beta}\mathfrak{G}_{ws_{i_1}\cdots s_{\alpha+1}}$. Let $\gamma=\beta$ and $j = \alpha+1$. Then 
		\[\overline{\partial}_{j+1}(g) = \cdots = \overline{\partial}_{j+\gamma - 1}(g)= g\quad\mbox{and}\quad \partial_{j+1}(g) = \cdots = \partial_{j+\gamma - 1}(g)= 0.\]
	\end{lemma}
	\begin{proof}
		The choice of $j$ and $\gamma$ make $j+1=\alpha+2$ and $j+\gamma-1=i_1$.
		Since $w$ is sorted, $w_{\alpha+1}<\cdots<w_{i_1}$. Then $ws_{i_1}\cdots s_{\alpha+1}$ has ascents at positions $\alpha+2,\ldots,i_1$. Fix $k\in [\alpha+2,i_1]$. By Proposition \ref{prop:repeatoperators}, we obtain
		\[x_{\alpha+1}^\beta\overline{\partial}_k(g)=\overline{\partial}_k(x_{\alpha+1}^\beta g)=\overline{\partial}_k(\mathfrak{G}_{ws_{i_1}\cdots s_{\alpha+1}})=\mathfrak{G}_{ws_{i_1}\cdots s_{\alpha+1}}\]
		and
		\[x_{\alpha+1}^\beta\partial_k(g)=\partial_k(x_{\alpha+1}^\beta g)=\partial_k(\mathfrak{G}_{ws_{i_1}\cdots s_{\alpha+1}})=0.\]
		Thus,
		\[\overline{\partial}_{k}(g)= g\quad\mbox{and}\quad \partial_{k}(g)= 0.\qedhere\]
	\end{proof}
	
	\begin{theorem}
		\label{thm:groth-equality}
		For any $w\in S_n$, $\mathscr{G}_w=\mathfrak{G}_w$.
	\end{theorem}
	\begin{proof}
		Start by extending the orthodontic sort order $\leq_{\mathrm{os}}$ to a linear order $L$ on $S_n$, viewed as a bijection $L:[n!]\to S_n$. We prove by induction on $j$ that 
		\[\mathscr{G}_{L(j)} = \mathfrak{G}_{L(j)}. \]
		For the base case $j=1$, Proposition \ref{prop:partialorder} implies $L(1)$ is the identity permutation $\mathrm{id}\in S_n$. It is easy to check 
		\[\mathscr{G}_{\mathrm{id}}=1=\mathfrak{G}_{\mathrm{id}}. \]
		
		Now, assume that $\mathscr{G}_{L(j')} = \mathfrak{G}_{L(j')}$ for all $j'<j$. Set $w=L(j)$, and let $(h,C,\alpha,i_1,\beta)$ be the primary column data of $w$. We first dispense with the case that $w$ is not a sorted permutation. 
		
		Suppose that $\sigma(w)$ has Rothe diagram equal to the Young diagram of $\lambda=(\lambda_1,\ldots,\lambda_{\beta})$. It follows from Proposition \ref{prop:unsortedorthodontia} that
		\[\mathscr{G}_w=x_{\alpha+1}^{\lambda_1}\cdots x_{i_1}^{\lambda_{\beta}} \mathscr{G}_{w_\sort}. \]
		The defining relation (\ref{eqn:daggerx1}) of $\leq_{\mathrm{os}}$ implies that $w_\sort<_{\mathrm{os}} w$. Then by induction,
		\[\mathscr{G}_{w_\sort}=\mathfrak{G}_{w_\sort}. \] 
		Applying Proposition \ref{prop:grothunsort} yields
		\begin{align*}
			x_{\alpha+1}^{\lambda_1}\cdots x_{i_1}^{\lambda_{\beta}}\mathfrak{G}_{w_\sort} = \mathfrak{G}_w.
		\end{align*}
		This completes the case that $w$ is not sorted.
		
		Now, assume that $w$ is a sorted permutation. Parts $\textup{(i)}-\textup{(iv)}$ of Theorem \ref{thm:sortedorthodontia} imply
		\begin{align}
			\label{eqn:5-groth}
			\mathscr{G}_w = \omega_1^{k_1}\cdots\omega_\alpha^{k_\alpha}\omega_{i_1+1}^{k_{i_1+1}}\cdots \omega_n^{k_n}\overline{\pi}_{i_1}(\overline{\pi}_{i_1 - 1}(\cdots\overline{\pi}_{\alpha+1}(\omega_{\alpha+1}^{m_\beta}(\overline{\pi}_{i_{\beta+1}}(\cdots\overline{\pi}_{i_\l}(\omega_{i_\l}^{m_\l})\cdots)))\cdots)).
		\end{align}
		The fundamental weights $\omega_1, \ldots, \omega_\alpha, \omega_{i_1+1}, \ldots, \omega_n$ are fixed under the actions of $s_{i_1}, s_{i_1 - 1}, \ldots, s_{\alpha+1}$, so we may rewrite \eqref{eqn:5-groth} as
		\[
			\mathscr{G}_w = \overline{\pi}_{i_1}(\overline{\pi}_{i_1-1}(\cdots\overline{\pi}_{\alpha+1}(\omega_1^{k_1}\cdots\omega_\alpha^{k_\alpha}\omega_{\alpha+1}^{m_\beta}\omega_{i_1+1}^{k_{i_1+1}}\cdots \omega_n^{k_n}(\overline{\pi}_{i_{\beta+1}}(\cdots\overline{\pi}_{i_\l}(\omega_{i_\l}^{m_\l})\cdots))\cdots)\cdots)).
		\]
		
		On the other hand, part $\textup{(v)}$ of Theorem~\ref{thm:sortedorthodontia} asserts that
		\[
			\mathscr{G}_{ws_{i_1}\cdots s_{\alpha+1}} =\omega_1^{k_1}\cdots\omega_\alpha^{k_\alpha-\beta}\omega_{\alpha+1}^{\beta+m_\beta}\omega_{i_1+1}^{k_{i_1+1}}\cdots\omega_n^{k_n}(\overline{\pi}_{i_{\beta+1}}(\cdots\overline{\pi}_{i_\l}(\omega_{i_\l}^{m_\l})\cdots)).
		\]
		Thus, we obtain
		\begin{align*}
			\mathscr{G}_{w} = \overline{\pi}_{i_1}(\overline{\pi}_{i_1 - 1}(\cdots\overline{\pi}_{\alpha+1}(x_{\alpha+1}^{-\beta} \mathscr{G}_{ws_{i_1}\cdots s_{\alpha+1}})\cdots)).
		\end{align*}
		The defining relation (\ref{eqn:daggerx2}) of $\leq_{\mathrm{os}}$ implies that $ws_{i_1}\cdots s_{\alpha+1}<_{\mathrm{os}} w$. Then by induction,
		\[\mathscr{G}_{ws_{i_1}\cdots s_{\alpha+1}}=\mathfrak{G}_{ws_{i_1}\cdots s_{\alpha+1}}. \] 
		Consequently, we obtain
		\begin{equation}
			\label{eqn:6-groth}
			\mathscr{G}_w = \overline{\pi}_{i_1}(\overline{\pi}_{i_1-1}(\cdots\overline{\pi}_{\alpha+1}(x_{\alpha+1}^{-\beta}\mathfrak{G}_{ws_{i_1}\cdots s_{\alpha+1}})\cdots)).
		\end{equation}
		If $\beta=1$, then $i_1=\alpha+1$ and (\ref{eqn:6-groth}) reduces to
		\begin{align*}
			\mathscr{G}_w &= \overline{\pi}_{\alpha+1}(x_{\alpha+1}^{-1}\mathfrak{G}_{ws_{\alpha+1}}) 
			=\overline{\partial}_{\alpha+1}(\mathfrak{G}_{ws_{\alpha+1}})
			=\mathfrak{G}_w,
		\end{align*}
		completing the proof.
		
		Otherwise, $\beta\geq 2$. Set $g = x_{\alpha+1}^{-\beta}\mathfrak{G}_{ws_{i_1}\cdots s_{\alpha+1}}$. By Lemma \ref{lem:computational-groth5}, $g$ meets the assumptions of Lemma \ref{lem:computational-groth4} with $\gamma=\beta$ and $j=\alpha+1$. Hence,
		\[
			\mathscr{G}_w = \overline{\partial}_{i_1}\cdots\overline{\partial}_{\alpha+1}(\mathfrak{G}_{ws_{i_1}\cdots s_{\alpha+1}}).
		\]		
		Since $w$ is sorted, each permutation in the list
		\[w,\, ws_{i_1},\, w_{s_{i_1}s_{i_1-1}},\ldots,\, w_{s_{i_1}\cdots s_{\alpha+1}}\]
		covers the previous in the weak Bruhat order. Thus, the recursive definition of $\mathfrak{G}_w$ implies 
		\[
			\overline{\partial}_{i_1}\cdots\overline{\partial}_{\alpha+1}(\mathfrak{G}_{ws_{i_1}\cdots s_{\alpha+1}}) = \mathfrak{G}_w. \qedhere
		\]	
	\end{proof}
	
	As an immediate corollary, we obtain Theorem \ref{thm:groth-for-intro}, restated here for convenience.
	\begin{namedtheorem}[\ref{thm:groth-for-intro}]
		Let $w\in S_n$ have orthodontic sequence 
		\[\bm{i}(w) = (i_1,\ldots, i_\l),\quad \bm{k}(w)=(k_1, \ldots, k_n), \mbox{ and}\quad \bm{m}(w) = (m_1, \ldots, m_\l).\]
		Then 
		\[\mathfrak{G}_w = \omega_1^{k_1}\cdots\omega_n^{k_n}\overline{\pi}_{i_1}(\omega_{i_1}^{m_1}\overline{\pi}_{i_2}(\omega_{i_2}^{m_2}\cdots\overline{\pi}_{i_\l}(\omega_{i_\l}^{m_\l})\cdots)),
		\]
		where $\omega_i$ is the fundamental weight $\omega_i=x_1\cdots x_i$.
	\end{namedtheorem}
	
	By taking the lowest degree homogeneous component of both sides, we recover the orthodontia formula for Schubert polynomials.
	\begin{namedtheorem}[\ref{thm:magyaroperatortheorem}]\textup{(\cite[Proposition~15]{magyar})}
		Let $w\in S_n$ have orthodontic sequence 
		\[\bm{i}(w) = (i_1,\ldots, i_\l),\quad \bm{k}(w)=(k_1, \ldots, k_n), \mbox{ and}\quad \bm{m}(w) = (m_1, \ldots, m_\l).\]
		Then 
		\[\mathfrak{S}_w = \omega_1^{k_1}\cdots\omega_n^{k_n}\pi_{i_1}(\omega_{i_1}^{m_1}\pi_{i_2}(\omega_{i_2}^{m_2}\cdots\pi_{i_\l}(\omega_{i_\l}^{m_\l})\cdots)).
		\]
	\end{namedtheorem}

	\section{Application to the Degree and Support of Grothendieck Polynomials}
	\label{sec:application}
	
	In this section we present some consequences of Theorem \ref{thm:groth-for-intro} for the degree and support of Grothendieck polynomials. Little is known regarding the support of Grothendieck polynomials in general, though special cases have been addressed in \cite{genpermflows,symmgrothnewton}. Conjectures such as \cite[Conjecture 5.5]{MTY}, \cite[Conjecture 5.1]{genpermflows}, and \cite[Conjecture 22]{schlorentzian} currently remain open.
	
	There has been recent interest in combinatorial formulas for the degree of a Grothendieck polynomial, in part due to its connection to the Castelnuovo--Mumford regularity of certain varieties \cite{symmgrothdeg}. Theorem \ref{thm:groth-for-intro} immediately yields the following combinatorial upper bound for the degree.

	\begin{proposition}
		\label{prop:orthodegbound}
		Let $w\in S_n$ and $\bm{i}(w)=(i_1,\ldots,i_\l)$. Then
		\[\deg \mathfrak{G}_w\leq \deg \mathfrak{S}_w+\l \]
	\end{proposition}
	\begin{proof}
		Each operator $\overline{\pi}_j$ in Theorem \ref{thm:groth-for-intro} increases the degree of its input by at most one.
	\end{proof}
%	\begin{proposition}
%		\label{prop:upperclosurebound}
%		For any $w\in S_n$, 
%		\[\deg \mathfrak{G}_w\leq \#\overline{D(w)} \]
%	\end{proposition}
%	\begin{proof}
%		Suppose $\bm{i}(w)=(i_1,\ldots,i_\l)$. If $p=0$ then $\mathfrak{G}_w=\mathfrak{S}_w$ and $\overline{D}=D$, so there is nothing to prove. For $p>0$, note that each orthodontia step decreases the degree by at most one and decreases the number of boxes in the upper closure by at least one. \com{seems wrong}
%	\end{proof}

%	We now work towards refining Proposition \ref{prop:upperclosurebound}. 
	We propose a possible refinement of Proposition \ref{prop:orthodegbound} in Conjecture \ref{conj:refineddegbound}. We now work towards proving Theorem \ref{thm:upwardsdivisibility-intro}, a new divisibility restriction on the monomials that can appear in a Grothendieck polynomial. We then deduce Corollary \ref{cor:upperclosurebound}, another combinatorial upper bound on the degree of a Grothendieck polynomial.
	
	For any diagram $D$, denote by $\overline{D}$ the \emph{upper closure} of $D$, the diagram 
	\[\overline{D} = \{(i,j)\mid j=j' \mbox{ and } i\leq i' \mbox{ for some } (i',j')\in D   \}. \]
	Denote by $\bm{x}^D$ the monomial
	\[\bm{x}^D = \prod_{(i,j)\in D} x_i = \prod_{j=1}^n\prod_{i\in D_j}x_i, \]
	where $D_1,\ldots,D_n$ are the columns of $D$.
	\begin{lemma}
		\label{lem:exponentchange}
		Let $w\in S_n$ be a nonidentity sorted permutation with primary column data $(h,C,\alpha,i_1,\beta)$. 
%		Suppose $w$ has orthodontic sequence 
%		\[\bm{i}(w) = (i_1,\ldots, i_\l),\quad \bm{k}(w)=(k_1, \ldots, k_n), \mbox{ and}\quad \bm{m}(w) = (m_1, \ldots, m_\l).\]
		Let \[\gamma=\#\{j\in [h+1,n] \mid \max(D(w)_j)=i_1+1 \} \qquad(\mbox{\textup{taking} } \max(\emptyset)\coloneqq 0).\]
		The diagrams $\overline{D(w)}$ and $\overline{D(ws_{i_1}\cdots s_{\alpha + 1})}$ are related via the equation
		\[\bm{x}^{\overline{D(w)}} = \bm{x}^{\overline{D(ws_{i_1}\cdots s_{\alpha + 1})}}x_{\alpha+1}^{-\beta}x_{\alpha+2}^{\gamma}x_{\alpha+3}^{\gamma}\cdots x_{i_1+1}^{\gamma}.  \]
		Moreover, if $c_j$ denotes the exponent of $x_j$ appearing in $\bm{x}^{\overline{D(ws_{i_1}\cdots s_{\alpha + 1})}}x_{\alpha+1}^{-\beta}$, then
		\[c_{\alpha+1} = c_{\alpha+2}+\gamma=\cdots = c_{i_1+1}+\gamma. \]
	\end{lemma}
	\begin{proof}
		The argument is similar to that of Theorem \ref{thm:sortedorthodontia}. Denote $D(ws_{i_1}\cdots s_{\alpha + 1})$ by $D'$ for compactness. Let $D(w)$ have columns $D_1,\ldots,D_n$ and $D'$ have columns $D_1',\ldots,D_n'$. Then $\overline{D(w)}$ has columns $\overline{D_1},\ldots,\overline{D_n}$, and similarly for $D'$. 
		
		Since $w$ is sorted, the columns $D_{h-\beta+1},\ldots, D_{h}$ all equal $[\alpha]$, while the columns $D'_{h-\beta+1},\ldots, D'_{h}$ all equal $[\alpha+1]$. For $j\in [h-\beta]$, we have $D_j=D_j'$. The diagram $(D_{h+1}',\ldots,D_n')$ is obtained from the diagram $(D_{h+1},\ldots, D_n)$ by permuting the rows by $s_{i_1}\cdots s_{\alpha+1}$. For $j\geq h+1$, we have  $\overline{D_j}= \overline{D_j'}$ unless $\max(D_j)=i_1+1$. In this case, $\overline{D_j}=[i_1+1]$ and $\overline{D_j'}=[\alpha+1]$. There are exactly $\gamma$ such columns among $D_{h+1},\ldots,D_n$.
		
		We conclude
		\[\bm{x}^{\overline{D(w)}} = \bm{x}^{\overline{D'}}x_{\alpha+1}^{-\beta}x_{\alpha+2}^{\gamma}x_{\alpha+3}^{\gamma}\cdots x_{i_1+1}^{\gamma}.\]
		Since $w$ is sorted, the exponents of $x_{\alpha+1},\ldots,x_{i_1+1}$ in $\bm{x}^{\overline{D(w)}}$ are all equal. The last statement of the lemma follows immediately.
	\end{proof}
	
	\begin{example}
		Consider $w=923854761$, so $w$ is sorted with primary column data $h=3$, $C=\{1,4,5\}$, $\alpha=1$, $i_1=3$, and $\beta=2$. 
%The orthodontic sequence of $w$ is
%		\[\bm{i}(w) = (3,2,4,3,6,5,4,3),\quad \bm{k}(w)=(3,0,0,0,0,0,0,1,0), \mbox{ and}\quad \bm{m}(w) = (0,2,0,1,0,0,0,1).\]
		The diagrams of $w$ and $ws_3s_2$ are 
		\begin{center}
			\includegraphics[scale=.8]{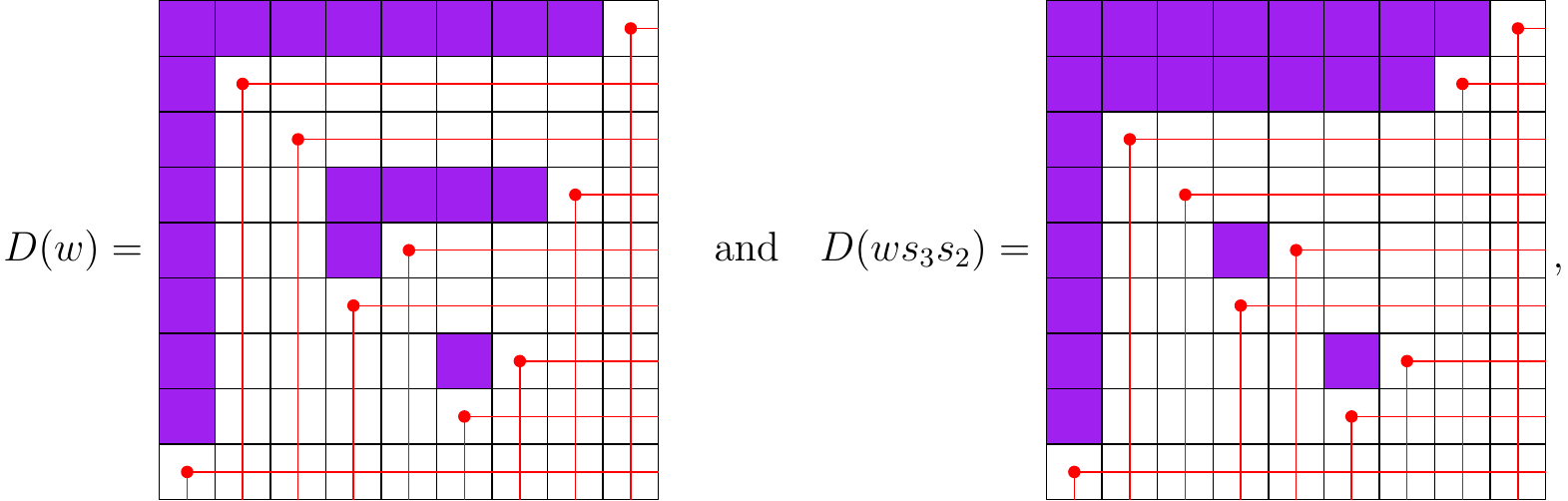}
		\end{center}
		with upper closures
		\begin{center}
			\includegraphics[scale=.8]{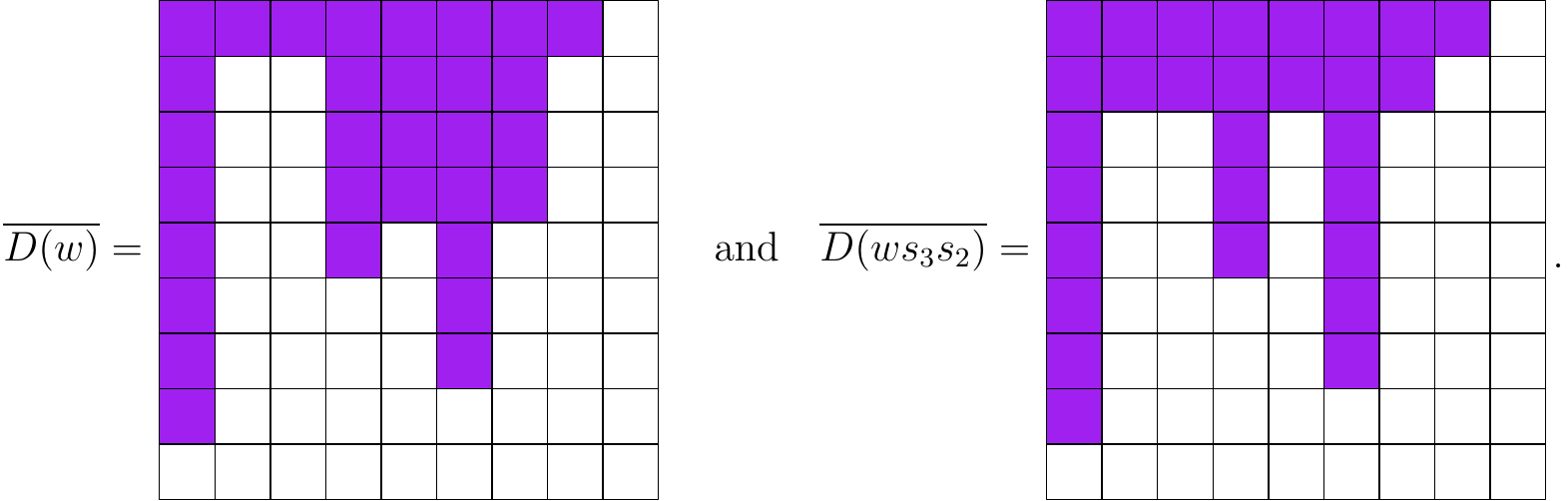}
		\end{center}
		Clearly,
		\[
			\bm{x}^{\overline{D(ws_{i_1}\cdots s_{\alpha + 1})}} = x_1^7x_2^7x_3^3x_4^3x_5^3x_6^2x_7^2x_8 \quad\mbox{and}\quad \bm{x}^{\overline{D(w)}} =x_1^8x_2^5x_3^5x_4^5x_5^3x_6^2x_7^2x_8 = x_2^{-2}x_3^2x_4^2\bm{x}^{\overline{D(ws_{i_1}\cdots s_{\alpha + 1})}}.
		\]
	\end{example}

	\begin{lemma}
		\label{lem:pijdivisibility}
		Let $X=x_1^{d_1}\cdots x_n^{d_n}$ be any monomial. Each monomial appearing in $\overline{\pi}_j(X)$ divides
		\[x_1^{d_1}\cdots x_{j-1}^{d_{j-1}} x_j^{\max(d_j,d_{j+1})}x_{j+1}^{\max(d_j,d_{j+1})} x_{j+1}^{d_{j+1}}\cdots x_n^{d_n}. \]
	\end{lemma}
	\begin{proof}
		We compute $\overline{\pi}_j(x_j^ax_{j+1}^b)$ in four separate cases over $(a,b)$.
		
		\noindent If $a>b$, then
		\[
			\overline{\pi}_j(x_j^ax_{j+1}^b)=\left(x_j^ax_{j+1}^b+x_j^{a-1}x_{j+1}^{b+1}+\cdots+x_j^bx_{j+1}^a\right)-\left(x_j^ax_{j+1}^{b+1}+x_j^{a-1}x_{j+1}^{b+2}+\cdots+x_j^{b+1}x_{j+1}^a\right).
		\]
		If $a<b-1$, then
		\[
		\overline{\pi}_j(x_j^ax_{j+1}^b)=\left(x_j^{a+1}x_{j+1}^{b}+x_j^{a+2}x_{j+1}^{b-1}+\cdots+x_j^{b}x_{j+1}^{a+1}\right)-\left(x_j^{a+1}x_{j+1}^{b-1}+x_j^{a+2}x_{j+1}^{b-2}+\cdots+x_j^{b-1}x_{j+1}^{a+1}\right).
		\]
		If $a=b$, then 
		\[
			\overline{\pi}_j(x_j^ax_{j+1}^b) =x_j^ax_{j+1}^b.
		\]
		If $a=b-1$, then 
		\[
			\overline{\pi}_j(x_j^ax_{j+1}^b) =x_j^{a+1}x_{j+1}^b.
		\]
		Note that in each case, all monomials occurring in $\overline{\pi}_j(x_j^ax_{j+1}^b)$ divide $x_j^{\max(a,b)}x_{j+1}^{\max(a,b)}$. 
		Since
		\[\overline{\pi}_j(X)=\biggr(\prod_{p\neq j,j+1} x_p^{d_p}\biggr)\overline{\pi}_j(x_j^{d_j}x_{j+1}^{d_{j+1}}),\]
		the lemma follows.
	\end{proof}

	\begin{namedtheorem}[\ref{thm:upwardsdivisibility-intro}]
		For any permutation $w\in S_n$, all monomials appearing in $\mathfrak{G}_w$ divide $\bm{x}^{\overline{D(w)}}$.
	\end{namedtheorem}
	\begin{proof}
		Begin by extending the orthodontic sort order $\leq_{\mathrm{os}}$ to a linear order $L$ on $S_n$, viewed as a bijection $L:[n!]\to S_n$. We prove the theorem by induction on $j$. For the base case $j=1$, Proposition \ref{prop:partialorder} implies $L(1)$ is the identity permutation and the theorem is trivial. 
		
		Assume that the theorem holds for all $j'<j$. Set $w=L(j)$, and let $(h,C,\alpha,i_1,\beta)$ be the primary column data of $w$. If $w$ is not sorted, then the defining relation (\ref{eqn:daggerx1}) of $\leq_{\mathrm{os}}$ shows that $w_{\sort} = L(j')$ for some $j'<j$. The theorem then follows immediately from Propositions \ref{prop:grothunsort} and \ref{prop:unsortedorthodontia}. 
		
		Suppose instead that $w$ is sorted. By Theorem \ref{thm:groth-for-intro} and Theorem \ref{thm:sortedorthodontia},
		\begin{align*}
		\mathfrak{G}_{w} = \overline{\pi}_{i_1}(\overline{\pi}_{i_1 - 1}(\cdots\overline{\pi}_{\alpha+1}(x_{\alpha+1}^{-\beta} \mathfrak{G}_{ws_{i_1}\cdots s_{\alpha+1}})\cdots)).
		\end{align*}
		The defining relation (\ref{eqn:daggerx2}) of $\leq_{\mathrm{os}}$ shows that $ws_{i_1}\cdots s_{\alpha+1}= L(j')$ for some $j'<j$. Thus, any monomial appearing in $x_{\alpha+1}^{-\beta}\mathfrak{G}_{ws_{i_1}\cdots s_{\alpha+1}}$ divides $x_{\alpha+1}^{-\beta}\bm{x}^{\overline{D(ws_{i_1}\cdots s_{\alpha+1})}}$. 
		
		Let $X'$ be any monomial appearing in $\mathfrak{G}_{w}$. Suppose $X'$ appears in $\overline{\pi}_{i_1}\cdots \overline{\pi}_{\alpha+1}(X)$, where $X=x_1^{d_1}\cdots x_n^{d_n}$ appears in $x_{\alpha+1}^{-\beta}\mathfrak{G}_{ws_{i_1}\cdots s_{\alpha+1}}$. Repeated application of Lemma \ref{lem:pijdivisibility} implies that $X'$ divides 
		\[x_1^{d_1}\cdots x_{\alpha}^{d_\alpha} x_{\alpha+1}^{M_{\alpha+1}}\cdots x_{i_1+1}^{M_{i_1+1}}x_{i_1+2}^{d_{i_1+2}}\cdots x_n^{d_n}, \]
		where each $M_p$ satisfies $M_p\leq \max(d_{\alpha+1},\ldots, d_{i_1+1})$.
		
		Let $c_p$ denote the exponent of $x_p$ appearing in $\bm{x}^{\overline{D(ws_{i_1}\cdots s_{\alpha + 1})}}x_{\alpha+1}^{-\beta}$ for each $p$. Since $X$ appears in $x_{\alpha+1}^{-\beta}\mathfrak{G}_{ws_{i_1}\cdots s_{\alpha+1}}$, it follows that $d_p\leq c_p$ for each $p$. Hence Lemma \ref{lem:exponentchange} implies
		\[\max(d_{\alpha+1},\ldots, d_{i_1+1})\leq \max(c_{\alpha+1},\ldots, c_{i_1+1})=c_{\alpha+1} \]
		since $c_{\alpha+1}=c_p+\gamma$ for any $p\in [\alpha+2,i_1+1]$ (with $\gamma$ as defined in Lemma \ref{lem:exponentchange}). Thus, $X'$ divides the monomial 
		\[x_1^{d_1}\cdots x_{\alpha}^{d_\alpha} x_{\alpha+1}^{M_{\alpha+1}}\cdots x_{i_1+1}^{M_{i_1+1}}x_{i_1+2}^{d_{i_1+2}}\cdots x_n^{d_n},\]
		which in turn divides the monomial 
		\begin{align*}
			x_1^{d_1}\cdots x_{\alpha}^{d_\alpha} x_{\alpha+1}^{c_{\alpha+1}}x_{\alpha+2}^{c_{\alpha+2}+\gamma}\cdots x_{i_1+1}^{c_{i_1+1}+\gamma}x_{i_1+2}^{d_{i_1+2}}\cdots x_n^{d_n} 
			&=\bm{x}^{\overline{D(ws_{i_1}\cdots s_{\alpha + 1})}}x_{\alpha+1}^{-\beta}x_{\alpha+2}^{\gamma}x_{\alpha+3}^{\gamma}\cdots x_{i_1+1}^{\gamma}\\
			&= \bm{x}^{\overline{D(w)}}.
		\end{align*}
		Hence, we have shown that $X'$ divides $\bm{x}^{\overline{D(w)}}$ as needed.
	\end{proof}

	\begin{corollary}
		\label{cor:upperclosurebound}
		For any $w\in S_n$, 
		\[\deg \mathfrak{G}_w\leq \#\overline{D(w)}. \]
	\end{corollary}

	\begin{example}
		Let $w=14532$, so 
		\begin{center}
			\includegraphics[scale=1]{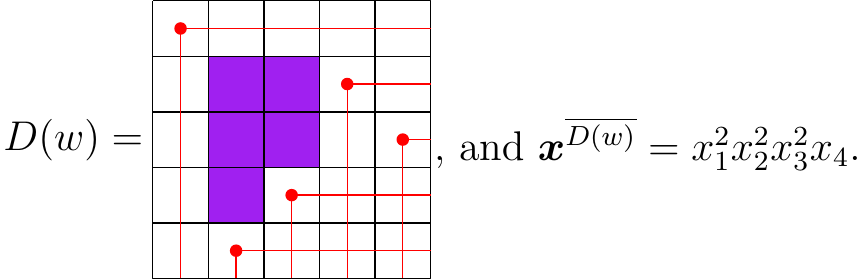}
		\end{center}
		Direct computation shows 
		\begin{align*}
			\begin{split}
				\mathfrak{G}_w=x_1^2x_2 x_3^2 + x_1^2x_2^2 x_3+x_1^2x_2^2 x_4+x_1^2x_3^2 x_4 +x_1^2x_2 x_3 x_4 +x_1x_2^2 x_3^2 +x_1x_2 x_3^2 x_4 +x_1x_2^2 x_3 x_4 +x_2^2 x_3^2 x_4\\
				-2 x_1^2x_2^2 x_3^2-3 x_1 x_2^2 x_3^2 x_4 -3 x_1^2 x_2^2 x_3 x_4 -3 x_1^2 x_2 x_3^2 x_4 +3 x_1^2 x_2^2 x_3^2 x_4.
			\end{split}
		\end{align*}
	\end{example}

	We conclude this section with a conjectural refinement of Proposition \ref{prop:orthodegbound} in the spirit of Theorem \ref{thm:upwardsdivisibility-intro}.
	\begin{definition}
		Let $w\in S_n$ and suppose $D(w)$ has columns $D_1,\ldots,D_n$. Taking $\max(\emptyset)\coloneqq0$, define $\theta(w)=(\theta_1,\ldots,\theta_n)$, where
		\[\theta_j=\#\{p\in[n]\mid j\leq \max(D_p) \} \mbox{ for each }j\in [n].\]
		For $\bm{i}(w)=(i_1,\ldots,i_\l)$, define $\xi(w)=(\xi_1,\ldots,\xi_n)$, where 
		\[\xi_j = \#\{p\in [\l]\mid i_p=j \}\mbox{ for each }j\in [n]. \]
	\end{definition}
	
	\begin{conjecture}
		\label{conj:refineddegbound}
		For any permutation $w\in S_n$, all monomials appearing in $\mathfrak{G}_w$ divide $\bm{x}^{\theta(w)+\xi(w)}$.
	\end{conjecture}
	
	\section{Strongly Separated Diagrams}
	\label{sec:stronglysep}
	In this final section, we briefly address the full generality in  which Magyar's formula applies. We define a general family of diagrams and explain how orthodontia assigns to each diagram a polynomial, similar to the case of Rothe diagrams.
	
	For $R,S\subseteq [n]$, we write $R\preccurlyeq S$ if $R$ is element-wise less than $S$. More precisely, $R\preccurlyeq S$ if $r\leq s$ for each $r\in R$ and each $s\in S$.
	
	\begin{definition}
		A diagram $D\subseteq [n]^2$ is \emph{strongly separated} if for every pair of columns $C,C'$ of $D$, either
		\[C\backslash C'\preccurlyeq C'\backslash C \quad \mbox{or}\quad C'\backslash C\preccurlyeq C\backslash C'. \]
	\end{definition}
	
	Whenever the columns $D_1,\ldots,D_n$ of a strongly separated diagram $D$ are ordered so that $D_i\backslash D_j \preccurlyeq D_j\backslash D_i$ whenever $i<j$, the \emph{orthodontic sequence}
	\[\bm{i}(D)=(i_1,\ldots,i_\l),\quad \bm{k}(D)=(k_1,\ldots,k_n),\quad \mbox{and}\quad \bm{m}(D)=(m_1,\ldots,m_\l)\] 
	is defined exactly as it was for Rothe diagrams of permutations (Definition \ref{def:ikmsequence}).
	
	\begin{example}
		Consider the diagrams $D$ and $D'$ given by
		\begin{center}
			\includegraphics[scale=1]{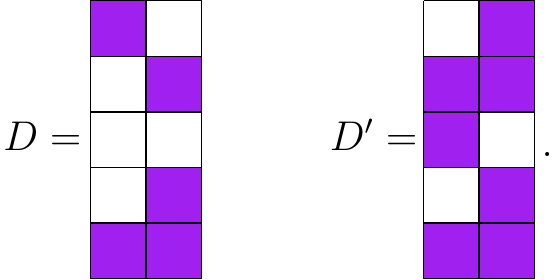}
		\end{center}
		The diagram $D$ is strongly separated, while $D'$ is not. Since the columns $D_1$ and $D_2$ of $D$ satisfy $D_1\backslash D_2\preccurlyeq D_2\backslash D_1$, the orthodontic sequence of $D$ is
		\[\bm{i}(D)=(4, 3, 2, 1, 2, 4, 3),\quad \bm{k}(D)=(0, 0, 0, 0, 0),\quad \mbox{and}\quad \bm{m}(D)=(0, 0, 1, 0, 0, 0, 1). \]
	\end{example}
	
	\begin{definition}
		Let $D\subseteq [n]^2$ be any strongly separated diagram with orthodontic sequence 
		\[\bm{i}(D) = (i_1,\ldots, i_\l),\quad \bm{k}(D)=(k_1, \ldots, k_n), \mbox{ and}\quad \bm{m}(D) = (m_1, \ldots, m_\l).\]
		Define 
		\[\mathscr{S}_D = \omega_1^{k_1}\cdots\omega_n^{k_n}{\pi}_{i_1}(\omega_{i_1}^{m_1}{\pi}_{i_2}(\omega_{i_2}^{m_2}\cdots{\pi}_{i_\l}(\omega_{i_\l}^{m_\l})\cdots))
		\]
		and
		\[\mathscr{G}_D = \omega_1^{k_1}\cdots\omega_n^{k_n}\overline{\pi}_{i_1}(\omega_{i_1}^{m_1}\overline{\pi}_{i_2}(\omega_{i_2}^{m_2}\cdots\overline{\pi}_{i_\l}(\omega_{i_\l}^{m_\l})\cdots)),
		\]
		where $\omega_i$ is the fundamental weight $\omega_i=x_1\cdots x_i$.
	\end{definition}
	
	Starting from the geometry of Bott--Samelson varieties, Magyar proves that the polynomials $\mathscr{S}_D$ are exactly the dual-characters of the flagged Weyl modules of strongly separated diagrams \cite[Corollary 13]{magyar}. In particular, $\mathscr{S}_D$ is a Schubert polynomial whenever $D$ is a Rothe diagram \cite{KP}, and a key polynomial whenever $D$ is left-aligned in each row \cite{keypolynomials}.
	
	What are the polynomials $\mathscr{G}_D$? Theorem \ref{thm:groth-for-intro} identifies them as the Grothendieck polynomials when $D$ is a Rothe diagram. It is easy to check that whenever $D$ is left-aligned, the polynomials $\mathscr{G}_D$ are the Lascoux polynomials \cite{grothtransition}, inhomogeneous analogues of the key polynomials. 
	
	Is there a K-theoretic analogue of the flagged Weyl module unifying these partial results?

	\section*{Acknowledgments} 
	We are grateful to Allen Knutson, Ricky Ini Liu, and Alex Yong for helpful discussions, and to Alex Fink for a careful reading.
	
	\bibliographystyle{plain}
	\bibliography{orthodontia-bibliography}
\end{document}